\title{
Unstable Interface Dynamics for Gravity Stokes Flow}
\author{Francisco Gancedo
, Rafael Granero-Belinchón, Zhongtian Hu, Elena Salguero, Yao Yao 
}

\documentclass[11pt]{article}
\usepackage[top=1in, bottom=1.in, left=1.in, right=1.in]{geometry}

 \usepackage[bookmarksnumbered, bookmarksopen, colorlinks, citecolor=red, linkcolor=black]{hyperref}
 \usepackage{amsfonts,stmaryrd}
\usepackage{graphicx}
 \usepackage{amsmath, amssymb}
 \usepackage{amsthm}
 \usepackage{mathtools}\usepackage{calc}
 \usepackage{verbatim}
 \usepackage[shortlabels]{enumitem}
 \usepackage{multicol}
\usepackage{tikz}
\usepackage{pgfplots}

\usepackage{accents}

\numberwithin{equation}{section}

\newtheorem{thm}{Theorem}[section]
\newtheorem{lem}[thm]{Lemma}
\newtheorem{cor}[thm]{Corollary}

\newtheorem{prop}[thm]{Proposition}
\newtheorem{rmk}[thm]{Remark}
\newtheorem{defn}[thm]{Definition}

\newcommand{\eps}{\epsilon}

\newcommand{\R}{\mathbb{R}}
\newcommand{\N}{\mathbb{N}}

\newcommand{\T}{\mathbb{T}}

\newcommand{\calS}{\mathcal{S}}
\newcommand{\calL}{\mathcal{L}}
\newcommand{\calK}{\mathcal{K}}

\newcommand{\dotz}{\dot{z}}
\newcommand{\p}{\partial}

\DeclareMathOperator{\supp}{supp}

\newcommand{\red}[1]{\textcolor{red}{#1}}

\definecolor{orange}{rgb}{1,.5,0}

\begin{document}

\maketitle
\begin{abstract}
We investigate some unstable behavior of the interface given by two incompressible fluids of different densities evolving by the regular Stokes law with gravity force. In the unstable scenario, where the denser fluid lies above the lighter fluid, we prove infinite-in-time growth of the length or the curvature of the interface. We support these analytical results with numerical simulations that confirm the predicted growth phenomena. In the stable configuration, where the denser fluid lies below the lighter fluid, we show that certain initial configurations evolve into the unstable regime in finite time. 
\end{abstract}

\section{Introduction}
In this paper, we study various unstable behavior of the sharp interface problem for the two-dimensional gravity Stokes flow. Indeed, we focus on the study of the two-phase Stokes system for incompressible, homogeneous fluids: 
\begin{equation}\label{eq:Stokes}
    \begin{cases}
        -\Delta u^{\pm} = -\nabla p^\pm - (0,\rho^\pm)^T,&\text{in }\Omega^\pm(t),\\
        \nabla\cdot u^\pm = 0,&\text{in }\Omega^\pm(t),\\
        \llbracket(\nabla u + \nabla u^T - p \, \text{Id})n\rrbracket = 0,&\text{on }\Gamma(t),\\
        \llbracket u \rrbracket = 0,&\text{on }\Gamma(t),\\
        z_t(\alpha,t) = u(z(\alpha,t),t),&\text{on }\Gamma(t),
    \end{cases}
\end{equation}
where $\rho^\pm$ are constant but different densities of the fluids located in the upper region 
$\Omega^+$ and the lower region $\Omega^-$, respectively. We consider a horizontally periodic flow, namely, we restrict ourselves to the case of
$$
\Omega^+(t)\cup\Omega^-(t)\cup\Gamma(t)=\mathbb{T}\times \mathbb{R}
$$
where $\T = [-\pi, \pi)$ is the standard torus and $\partial\Omega^{\pm}(t)=\Gamma(t)$ is the moving free boundary. Here, we also use the notation $\llbracket f \rrbracket = f^+ - f^-$ on $\Gamma(t)$. The free boundary is also denoted by the curve $z(\alpha,t)$ for some parametrization, with $\alpha \in \mathbb{T}$. The main motivation to study such a free boundary problem comes from the fact that at low Reynolds numbers, the viscosity forces dominate the inertial forces acting on the fluid. Thus, in a certain limiting regime, the fluids flow in the Stokes regime and their dynamics is captured by \eqref{eq:Stokes}. For more studies on steady viscous flow we refer to the pioneer work Ladyzhenskaya \& Solonnikov \cite{ladyzhenskaya1983determination} and the references therein.

The previous multiphase system can also be viewed as a sharp-interface case of the following, so called Stokes-Transport system \cite{hofer2025sedimentation}

\begin{equation}\label{eq:StokesTr}
    \begin{cases}
        -\Delta u= -\nabla p - (0,\rho)^T,&\text{in }\mathbb{T}\times \mathbb{R},\\
        \nabla\cdot u = 0,&\text{in }\mathbb{T}\times \mathbb{R},\\
\partial_t \rho+\nabla\cdot(u\rho)=0, &\text{in }\mathbb{T}\times \mathbb{R}.   \end{cases}
\end{equation}

This active scalar equation has received significant attention in recent years. In particular, this system was derived by Höfer \cite{hofer2018sedimentation}  and Mecherbet \cite{mecherbet2019sedimentation} as a model for sedimentation of particles in a viscous fluid, in the case of negligible inertia. The existence and uniqueness of solutions under certain integrability and regularity assumptions for the density have been proved in several recent works (e.g. \cite{cobb2023wellposednessfractionalstokestransport,dalibard2025long,grayer2023dynamics,lazar2025global,leblond2022wellposedness,mecherbet2021onthesedimentation}). For instance, Grayer II \cite{grayer2023dynamics} showed the global well-posedness for patch-like and $L^1 \cap L^\infty$ initial density, as well as regularity persistence of the boundary of the patch. More recently, it was proved in \cite{lazar2025global} that higher Hölder regularities are also propagated. Moreover,  Dalibard, Guillod \& Leblond \cite{dalibard2025long} proved global well-posedness for bounded initial density in the case of bounded domains and the strip $\Omega =  \mathbb{T} \times (0,1)$. We refer the reader to \cite{mecherbet2024afewremarks} for a detailed review of the literature and well-posedness results for densities in $L^1 \cap L^p$. Our setting differs to the previous frameworks in the sense that the density field is merely a piecewise constant function, so it lacks integrability in the chosen unbounded region.


In \cite{GGS25_SIMA} it was proved that one can equivalently write \eqref{eq:Stokes} in the form of the following Contour Dynamics Equation (CDE):
\begin{equation}
    \label{eq:cde}
    \p_t z(\alpha, t) = (\rho^- - \rho^+)\int_\T \calS(z(\alpha, t) - z(\beta, t)) \, \dotz^\perp(\beta, t)z_2(\beta,t) d\beta,\quad \alpha \in \T,
\end{equation}
where $\dotz(\beta) := \p_\beta z(\beta)$. Moreover, the kernel $\calS(x)$ is the $x_1$--periodic Stokeslet, which can be explicitly written as 
$$
\calS(x_1,x_2) = \frac{1}{8\pi}\log(2(\cosh(x_2) - \cos(x_1))) \, \text{Id} - \frac{x_2}{8\pi(\cosh(x_2) - \cos(x_1))}\begin{bmatrix}
    -\sinh(x_2) & \sin(x_1)\\
    \sin(x_1) & \sinh(x_2)
\end{bmatrix}.
$$
Although for gravity $x_1-$periodic flows such contour dynamics approach was new, an integral approach has been used before for the case of capillarity-driven Stokes interfaces in the whole 2D plane. In this regard, the interested reader is referred for instance to the papers by Badea \& Duchon and Matioc \& Prokert \cite{badea1998capillary,matioc2021two} (see also \cite{matioc2022two,matioc2023capillarity}). More recently, the combined case of gravity and capillarity forces was studied in \cite{bohme2025wellposedness,böhme2025wellposednessrayleightaylorinstabilitytwophase} in the horizontally-periodic case. See also \cite{JaeHoChoi} for the capillarity-driven stability of nearly circular closed curves in the absence of gravity.

A remarkable fact about the problem \eqref{eq:cde} is that the equation is globally well-posed, as long as we impose the initial regularity assumption $z_0 \in C^{k,\gamma}(\T)$ for arbitrary $k\in\N\cup\{0\}$ and $0<\gamma< 1$, as shown in \cite{GGS25}. In particular, this global well-posedness result holds regardless of the size of the initial data or the initial stratification. Namely, the solution exists globally whether the denser fluid lies above or below the lighter one. Consequently, this global existence result is independent of the Rayleigh–Taylor condition. This behavior stands in sharp contrast with that of the two-phase Muskat problem, a closely related active scalar equation, which can in fact be viewed as a more singular counterpart of \eqref{eq:cde}. 

Moreover, if one restricts to the behavior of \eqref{eq:cde} near stably stratified states, \cite{GGS25_SIMA} shows that perturbations of such states are asymptotically stable in both Sobolev spaces and in certain spaces of analytic functions. These stability results are based on the fact that, on the linear level, \eqref{eq:cde} exhibits a weak damping effect where the linear operator is given by the Neumann-to-Dirichlet operator (instead of the more common Dirichlet-to-Neumann operator appearing in the linearized version of both the Muskat, the water waves and even the capillary Stokes problems). On the other hand, in \cite{GGS25_SIMA} it is shown that unstable solutions grow. This instability result roughly reads as follows: fix an arbitrary $T$, then there exists an initial interface that grows exponentially at least for such $T$. Due to the way these solutions are constructed, both, the initial interface and its corresponding solution at time $T$ are in fact small in certain sense.

However, the fact that the solution exists globally regardless of the size of the initial data and the sign of the density jump raises the following rather big question: \emph{what is the dynamics of these globally defined solutions?} Some possible \emph{a priori} answer, for instance, would be that the solutions stabilize around particular stable, steady, non-flat interfaces. A different possible answer would be that every solution keeps growing (in a certain sense) for arbitrary large times. It is in this question that we focus on. Indeed, with the discussion above, the exploration of both instabilities and finite-time singular behavior of the system \eqref{eq:Stokes} is rather limited, both when the initial configuration is stable ($\rho^+ > \rho^-$) and unstable ($\rho^+ > \rho^-$). Hence, in this paper, we concern ourselves with the study of possible instability scenarios in both stable and unstable regimes. 

Our first main result aims to capture long-time instabilities for initial data belonging to the unstable regime. To the best of the authors' knowledge, there are no rigorous results discussing the long-time behavior of the solutions to such initial data, except for the preservation of boundary regularity as shown in \cite{GGS25}. In the first main result of this paper, we show that the infinite-in-time growth of some key geometric quantities of the free interface is rather generic given initial data in the unstable regime. More precisely, we denote the the perimeter of the interface $\Gamma_t$ as
\[
L(t) := |\Gamma_t| = \int_\T |\dot{z}(\beta,t)| d\beta
\]
and the maximal curvature of the interface $\Gamma_t$ as
\[
\calK(t) := \max_{x \in \Gamma_t}|\kappa(t)|.
\]
This result can be stated as follows: for an initial interface $\Gamma_0$ under certain symmetry conditions, either the perimeter or the maximal curvature of the evolved interface $\Gamma_t$ must grow infinitely in time with some explicit power-law rates. 
\begin{thm}\label{thm:main}
    Let $\rho(x,t)$ be the global-in-time solution to an unstable initial datum $\rho_0 \not\equiv \rho_s$, where 
    $$
    \rho_s(x_2) = \begin{cases}
        1, & x_2 > 0,\\
        -1, & x_2 < 0,
    \end{cases}
    $$
    is the unstable stratified state. Moreover, suppose that $\Gamma_0$ is centrally symmetric and evenly symmetric with respect to $x_1 = \pm\frac{\pi}{2}$. Then for any $\eps > 0$, we have
    \begin{equation}
        \label{est:main}
        \limsup_{t \to \infty} t^{-{\frac{1}{17}} + \eps} (L(t) + \calK(t))= \infty.
    \end{equation}
\end{thm}

\begin{figure}[h]
    \centering
    \includegraphics{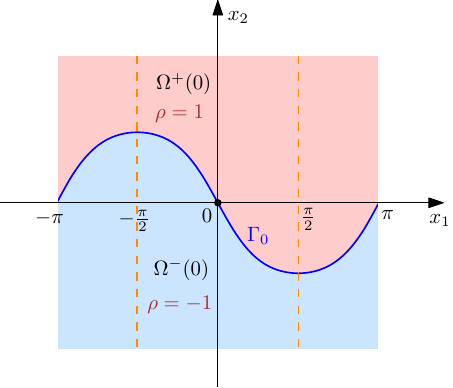}
    \label{fig:init}
    \caption{An example of an unstable, un-stratified configuration satisfying our symmetry assumptions in Theorem~\ref{thm:main}. Note that $\Gamma_0$ is centrally symmetric about the point $(0,0)$, and evenly symmetric with respect to the  orange dashed lines $x_1 = \pm\frac{\pi}{2}$.}
\end{figure}

\begin{rmk}
    \begin{enumerate}
        \item The growth phenomenon described in Theorem \ref{thm:main} qualitatively agrees with the numerical simulations reported in Section \ref{sec:numerics}, which contain evolutions displaying large growth in the length of the interface and boundedness in interface's curvature. These numerical evidences indeed suggest the validity of the dichotomy in Theorem \ref{thm:main}. 
        
        \item We do not expect the growth rate in \eqref{est:main} to be sharp. In fact, the numerical examples in Section \ref{sec:numerics} seem to suggests a linear rate of growth, which is much faster than the sublinear rate obtained in our main theorem.
    \end{enumerate}
\end{rmk}

The proof of Theorem \ref{thm:main} crucially exploits a natural potential energy possessed by the Stokes-Transport system, which is monotone as the solution evolves in time (See Section~\ref{sec_energy} for derivation). The majority of the analysis will focus on the analysis of this potential energy, where we obtain inequalities that connect its time derivative with the perimeter and maximal curvature of the evolved surface $\Gamma(t)$. We remark that the analysis of monotone quantities is also important in proving large-time growth phenomena and creation of small scales for other fundamental incompressible fluid models. We refer the readers to \cite{drivas2024twisting,kiselev2024small,kiselev2023small,park2024growth} and references therein.

We also highlight a connection between the growth phenomena found in Theorem \ref{thm:main} and the formation of viscous fingering, which occurs at the interface separating two immiscible fluids in a porous medium driven by gravity and is highly related to the well-known Rayleigh-Taylor instability. We refer the readers to \cite{menon2005dynamic,otto1997viscous} for a more thorough discussion about growth rate of fingers in this regime. In this regard, we give the second main result of this work, where we obtain an estimate on the number of {fingers} for some class of solutions given that certain {relaxation} in the number of {fingers} is observed in the numerics (see Section \ref{sec:numerics}).

A precise statement concerning the number of fingers is as follows:
\begin{thm}\label{thm:2}
Fix $T>0$ and $\mu>0$. There exists a family of initial interfaces $(\alpha,g_0(\alpha))$ such that their corresponding solutions to the Stokes system $(\alpha,g(\alpha,t))$ satisfy that they are analytic in the strip $\{\alpha+i\beta\text{ with } |\beta|\leq \nu^*\}$. Then $[-\pi, \pi]=I_\mu\cup R_\mu$, where $I_\mu$ is an union of at most $[\frac{4\pi}{\nu^*}]$ intervals open in $[-\pi, \pi]$, and
\begin{itemize}
\item $|\partial_\alpha g(\alpha,t)| \leq \mu, \text{ for all }\alpha\in I_\mu,$
\item $\text{Card}\{\alpha \in R_\mu : \partial_\alpha g(\alpha,t)=0\}\leq
F(t),$
\end{itemize}
where the growth of $F(t)$ is bounded from below as follows
$$
C(\nu^*,\mu,g_0)(1+\sqrt{t})\leq F(t).
$$
\end{thm}

It should be noted that we are not able to say anything on the possible relaxation on the number of fingers. Instead we can say that, in the case of maximal growth on the number of fingers, such maximal growth is at least as $\sqrt{t}$.

Our third main result concerns finite-time singular behavior in the stable regime. In particular, we prove that there exist initial smooth curves in the stable stratification of the densities that exhibit a turning instability, that is, the interface fails to be a graph in finite time, and the Rayleigh-Taylor condition breaks down. The proof follows \cite{castro2012rayleigh} but the operators involved in our proof are more regular. This  turning phenomenon is described in the following result: 
\begin{thm} \label{thm:turning}
    Let $\bar{z}(\alpha) = (\bar{z}_1(\alpha),\bar{z}_2(\alpha))$ a curve such that 
    \begin{enumerate}
        \item $\bar{z}_i$ are smooth and have odd symmetry.
        \item $\partial_\alpha \bar{z}_1(\alpha) >0$ for $\alpha \neq 0$, $\partial_\alpha \bar{z}_1(0) =0$ and $\partial_\alpha \bar z_2(0) >0$.
        \item $\partial_\alpha \bar u_1(0) = \partial_\alpha \partial_t \bar{z}_1(0) <0$.
    \end{enumerate}
    Then, if we set $\bar{z}(\alpha)$ as the initial datum for the Contour Dynamics Equation \eqref{eq:cde} in the Rayleigh-Taylor stable case of the densities, the evolution of $\bar{z}(\alpha)$ exhibits a turning instability at $\alpha =0$ in finite time.
    Moreover, we prove that there exists a family of initial data $\bar{z}$ fulfilling the conditions 1-3 and we construct another family of curves that satisfy the extra symmetry assumption:
    \begin{enumerate}
      \setcounter{enumi}{3}
        \item $\bar{z}_i$ are symmetric with respect to $\frac{\pi}{2}$ in the sense of \eqref{evensym}.
    \end{enumerate}
    For the latter family of curves fulfilling conditions 1-4, we prove that the turning instability occurs at $\alpha = 0$ and the endpoints $\alpha = \pm \pi$ due to the even symmetry of $\bar{z}_2$.
\end{thm}
\begin{figure}[h]
    \centering
    \includegraphics[width=0.5\textwidth]{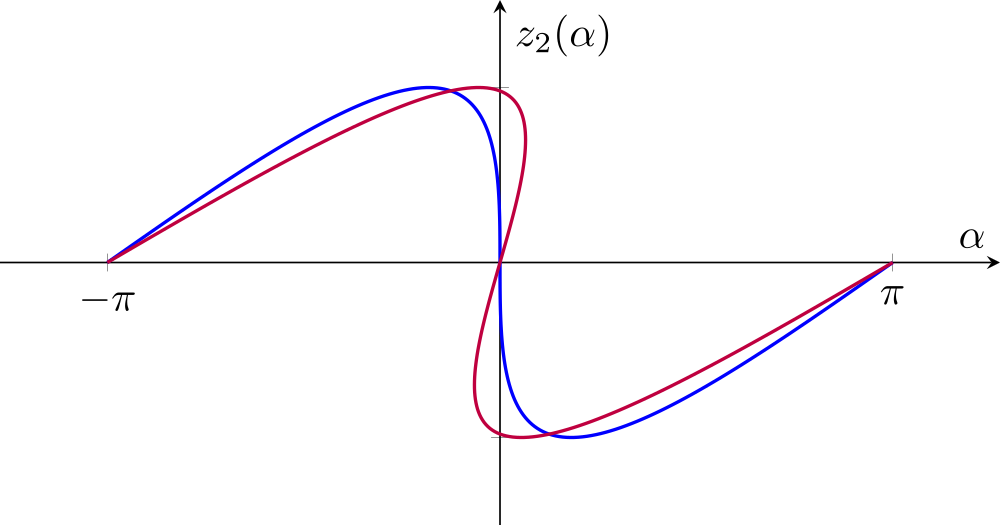}
    \caption{Sketch of the turning instability at $\alpha = 0$. The blue curve represents the initial interface and the red curve represents its evolution into a partially unstable regime.}
\end{figure}
\begin{rmk}
In the above result, we construct a curve that fulfills the symmetry assumptions needed in Theorem \ref{thm:main}. However, we cannot conclude that after the Rayleigh-Taylor condition breaks down, the dynamics turns to the scenario in Theorem \ref{thm:main}. In the latter, we need an infinite mass of the denser fluid on top, while in the turning instability, there is just a finite mass that enters in the unstable regime. It is not clear what the behavior of the fluid will be after the turning instability: it might become stable again (see \cite{CGZ17}) or develop infinite growth similar to the one shown in Theorem \ref{thm:main}.
\end{rmk}
\begin{rmk}
Theorem \ref{thm:turning} also holds for the Rayleigh-Taylor unstable case of the densities. In fact, in the unstable case, the conditions for the initial data to develop a turning instability are expected to be weaker. It is direct to check that the same choice of $\bar{z}_1$ as in the proof of Theorem \ref{thm:turning} and the simple choice $\bar{z}_2(\alpha) = \sin(\alpha)$ is enough to ensure turning in the unstable scenario.
\end{rmk}

We conclude the introduction by laying out the organization of this article. In Section \ref{sec:prelim}, we discuss several important properties concerning the Stokes-Transport system, including a class of symmetries and a potential energy. Section \ref{sec:finiteE} is dedicated to proving Theorem \ref{thm:main} on infinite-in-time growth in the unstable regime. In Section \ref{sec:relaxation}, we show Theorem \ref{thm:2} concerning the fingering phenomena. In Section \ref{sec:RTbreakdown}, we prove Theorem \ref{thm:turning} on Rayleigh-Taylor breakdown in the stable regime. Finally in Section \ref{sec:numerics}, we provide some numerical examples regarding the growth phenomena proved in Theorem \ref{thm:main} and \ref{thm:2}.

\section{Preliminaries on Stokes-Transport System}\label{sec:prelim}
In this section, we introduce several fundamental properties regarding \eqref{eq:cde}. We start with defining an equivalent weak formulation to the contour dynamics \eqref{eq:cde}. We then define a crucial symmetry property which is conserved by the evolution of the Stokes-Transport system that will be important for showing our first main result. We finally conclude this section by introducing a natural (modulated) potential energy tied intimately with the weak formulation and discussing several properties regarding this energy. For the rest of this section, we will always consider initial datum $z_0$ for \eqref{eq:cde} that has spatial regularity $C^{2,\gamma}$, for some $\gamma \in (0,1)$. By \cite[Corollary 1]{GGS25}, the corresponding solution $z(\cdot, t)$ is global in time as a $C^{2,\gamma}$ curve. Finally, we may without loss of generality consider $\rho^+ = 1$ and $\rho^- = -1$.
\subsection{An Equivalent Weak Formulation}
We start with introducing an appropriate weak formulation to the 2-phase Stokes-Transport system \eqref{eq:Stokes}. Let
\begin{equation}\label{defn:rho}
\rho(x,t) := \begin{cases}
    1, & x \in \Omega^+(t)\\
    -1, & x \in \Omega^-(t)
\end{cases},\quad
\rho_0(x) := \begin{cases}
    1, & x \in \Omega^+(0)\\
    -1, & x \in \Omega^-(0)
\end{cases}.
\end{equation}
We introduce the following definition of a weak solution:
\begin{defn}\label{defn:weak}
    We say that the pair $(\rho(x,t),u(x,t))$ is a weak solution to \eqref{eq:Stokes} if $\rho(x,t)$ defined by \eqref{defn:rho} solves the continuity equation
\begin{equation}
    \label{eq:Stokes2mass}
    \begin{split}
        \p_t \rho + u\cdot \nabla \rho = 0,&\quad x \in \T \times \R,\\
        \rho(x,0) = \rho_0(x).
    \end{split}
\end{equation}
in the sense of distributions, where
\begin{equation}\label{eq:Stokes2BS}
u = \nabla^\perp \Delta^{-2}\p_1 \rho.
\end{equation}
Here, $\nabla^\perp = (-\p_2, \p_1)$, and $\Delta^{-2}$ denotes the inverse of the bi-Laplacian operator on the domain $\T \times \R$. 
\end{defn}
We conclude by showing the following equivalence between the above weak formulation and the contour dynamics \eqref{eq:cde}.
\begin{prop}
    \label{prop:equivalence}
    Let $z(x,t)$ be a $C^{k,\gamma}$ solution to \eqref{eq:cde} with $k \in\N$, and $\gamma \in (0,1)$. Then $\rho(x,t)$ defined by \eqref{defn:rho} and $u(x,t)$ defined by \eqref{eq:Stokes2BS} is a weak solution in the sense of Definition \ref{defn:weak}. Conversely, if $(\rho,u)$ is a weak solution in the sense of Definition \ref{defn:weak} with $\Gamma(t)$ a $C^{k,\gamma}$ curve  with $k\in\N$, and $\gamma \in (0,1)$, then $z(x,t)$ is a $C^{k,\gamma}$ solution to \eqref{eq:cde}.
\end{prop}
The proof of the above proposition follows from a straightforward adaptation of Proposition 3.2 in \cite{cordoba2007contour}. We thus omit the details here.

\begin{rmk}
    Since we work with configurations with sufficiently regular free interface, thanks to Proposition \ref{prop:equivalence}, we will switch frequently between the Contour Dynamics formulation \eqref{eq:cde} and the weak formulation given in Definition \ref{defn:weak}, whichever is more convenient. Thus for the rest of the article, we will use this equivalence without further explanations.
\end{rmk}

\subsection{Symmetries}
This section is dedicated to showing two important symmetries of \eqref{eq:cde} which are conserved under the Stokes-Transport system. The proofs of the conservation of such symmetries are straightforward, but we are not able to pinpoint an exact reference in previous literature. Therefore, we provide detailed proofs here for the sake of completeness. Note that even though our main result concerns $C^{2,\gamma}$ interface, we prove the following statements for general $C^{1,\gamma}$ solutions to achieve the greatest generality.
\begin{prop}[Conservation of Central Symmetry]\label{prop:central}
    Fix $\gamma \in (0,1)$. Suppose that $z_0 \in C^{1,\gamma}$ is the initial datum for \eqref{eq:cde}, and that $z_0(\alpha) = -z_0(- \alpha)$ for any $\alpha \in \T$. Then $z(\alpha, t) = -z(- \alpha, t)$ for arbitrary $t > 0$, where $z(\cdot, t)$ is the (global) solution to \eqref{eq:cde} with initial datum $z_0$.
\end{prop}
\begin{proof}
    Consider $\tilde{z}(\alpha, t) := -z(-\alpha, t)$. In view of the uniqueness statement in \cite[Theorem 1]{GGS25}, the proof is completed once we show that $\tilde{z}$ is also a solution to \eqref{eq:cde} with initial datum $z_0$. Since
    $$
    \p_t \tilde{z}(\alpha, t) = -\p_t z(- \alpha, t),\quad \dot{\tilde{z}}(\alpha, t) = \dotz(- \alpha, t),
    $$
    we compute that
    \begin{align*}
        \p_t \tilde{z}(\alpha, t) &= 2\int_\T \calS(z(- \alpha, t) - z(\beta, t))\cdot \dotz^\perp(\beta, t)z_2(\beta,t) d\beta\\
        &= 2\int_\T \calS(z( - \alpha, t) - z( - \beta, t))\cdot \dotz^\perp( - \beta, t)z_2(- \beta,t) d\beta\\
        &= -2\int_\T \calS(-\tilde{z}(\alpha, t) + \tilde{z}(\beta, t))\cdot \dot{\tilde{z}}^\perp(\beta, t)\tilde{z}_2(\beta,t) d\beta\\
        &= -2\int_\T \calS(\tilde{z}(\alpha, t) - \tilde{z}(\beta, t))\cdot \dot{\tilde{z}}^\perp(\beta, t)\tilde{z}_2(\beta,t) d\beta,
    \end{align*}
    where we used the fact that $\calS(z) = \calS(-z)$ in the final equality. The proof is thus completed.
\end{proof}
Next, we show that if we further impose the following \textit{even symmetry} to the initial datum, namely for $z_0 = (z_{0,1}, z_{0,2})$:
\begin{equation}\label{evensym}
\begin{split}
    z_{0,1}(\alpha) &= -\pi - z_{0,1}(-\pi - \alpha),\quad z_{0,2}(\alpha) = z_{0,2}(-\pi - \alpha),\quad \alpha \in (-\pi,0],\\
    z_{0,1}(\alpha) &= \pi - z_{0,1}(\pi - \alpha),\quad z_{0,2}(\alpha) = z_{0,2}(\pi - \alpha),\quad \alpha \in (0,\pi],
\end{split}
\end{equation}
this symmetry will also be preserved along the evolution. In particular, we show the following Proposition:
\begin{prop}[Conservation of Central + Even Symmetry]
    \label{prop:even}
    Suppose $z(x,t)$ is the solution to \eqref{eq:cde} with initial data $z_0 \in C^{1,\gamma}$, $\gamma \in (0,1)$. Assume further that $z_0$ enjoys both central symmetry and \eqref{evensym}. Then for any $t \ge 0$,
\begin{equation}\label{evensym1}
\begin{split}
    z_{1}(\alpha,t) &= -\pi - z_{1}(-\pi - \alpha,t),\quad z_{2}(\alpha,t) = z_{2}(-\pi - \alpha,t),\quad \alpha \in (-\pi,0],\\
    z_{1}(\alpha,t) &= \pi - z_{1}(\pi - \alpha,t),\quad z_{2}(\alpha,t) = z_{2}(\pi - \alpha,t),\quad \alpha \in (0,\pi].
\end{split}
\end{equation}
\end{prop}
\begin{proof}
    For the rest of the proof, we suppress the temporal variable $t$ for the sake of clarity. Let $z(\alpha)$ be the solution initiated by data $z_0$, and define $\bar{z}(\alpha)$ by
    $$
    \bar{z}_1(\alpha) = \begin{cases}
        -\pi - z_1(-\pi - \alpha),&\alpha \in (-\pi,0],\\
        \pi - z_1(\pi-\alpha),&\alpha \in (0,\pi],
    \end{cases}
    $$
    and
    $$
    \bar{z}_2(\alpha) = \begin{cases}
        z_2(-\pi - \alpha),&\alpha \in (-\pi,0],\\
        z_2(\pi-\alpha),&\alpha \in (0,\pi].
    \end{cases}
    $$
    Similar to the proof of Proposition \ref{prop:central}, it suffices to show that $\bar{z}$ solves the CDE \eqref{eq:cde}. If $\alpha \in (-\pi,0]$, we have
    \begin{align*}
    \p_t \bar{z}_1(\alpha) &= -\p_tz_1(-\pi-\alpha) = \frac{1}{4\pi}\int_\T \left[\calS_{11}(-\pi-\alpha,\beta;z)(-\dotz_2(\beta)) + \calS_{12}(-\pi -\alpha,\beta;z)\dotz_1(\beta)\right]z_2(\beta)d\beta\\
    &= \frac{1}{4\pi}\left(\int_{-\pi}^0 + \int_0^\pi \right)\left[\calS_{11}(-\pi-\alpha,\beta;z)(-\dotz_2(\beta)) + \calS_{12}(-\pi -\alpha,\beta;z)\dotz_1(\beta)\right]z_2(\beta)d\beta\\
    &= \frac{1}{4\pi}(I_1 + I_2),
    \end{align*}
    where
    \begin{align*}
    \calS_{11}(\alpha,\beta;z) = \log\left(2(\cosh(z_2(\alpha) - z_2(\beta))- \cos(z_1(\alpha) - z_1(\beta)))\right)\\
    + \frac{(z_2(\alpha) - z_2(\beta))\sinh(z_2(\alpha) - z_2(\beta))}{\cosh(z_2(\alpha) - z_2(\beta)) - \cos(z_1(\alpha) - z_1(\beta))},
    \end{align*}
    $$
    \calS_{12}(\alpha,\beta;z) = - \frac{(z_2(\alpha)-z_2(\beta))\sin(z_1(\alpha) - z_1(\beta))}{\cosh(z_2(\alpha) - z_2(\beta)) - \cos(z_1(\alpha) - z_1(\beta))}.
    $$
    Applying the change of variable $\bar\beta = -\pi - \beta$ and that
    $$
    z_1(-\pi-\alpha) - z_1(-\pi - \bar\beta) = \bar{z}_1(\bar\beta) - \bar{z}_1(\alpha),\quad
    z_2(-\pi - \alpha) - z_2(-\pi - \bar\beta) = \bar{z}_2(\alpha) - \bar{z}_2(\bar\beta),
    $$
    $$
    \p_{\bar\beta}\bar{z}_1(\bar\beta) = \p_\beta z_1(\beta),\quad \p_{\bar\beta}\bar{z}_2(\bar\beta) = -\p_\beta z_2(\beta),
    $$
    we have
    \begin{align*}
        I_1 &= -\int_{-\pi}^0 \left[ \calS_{11}(\alpha,\bar\beta;\bar{z})(-\p_{\bar\beta}\bar{z}_2(\bar\beta)) + \calS_{12}(\alpha,\bar\beta;\bar{z})\p_{\bar\beta}\bar{z}_1(\bar\beta) \right] \bar{z}_2(\bar\beta) d\bar\beta.
    \end{align*}
    A similar argument yields
    $$
    I_2 = -\int_0^\pi \left[ \calS_{11}(\alpha,\bar\beta;\bar{z})(-\p_{\bar\beta}\bar{z}_2(\bar\beta)) + \calS_{12}(\alpha,\bar\beta;\bar{z})\p_{\bar\beta}\bar{z}_1(\bar\beta) \right] \bar{z}_2(\bar\beta) d\bar\beta.
    $$
    Combining the computations above, we arrive at
    \begin{equation}\label{eq:evenaux1}
    \p_t\bar{z}_1(\alpha) = -\frac{1}{4\pi}\int_\T \left[ \calS_{11}(\alpha,\bar\beta;\bar{z})(-\p_{\bar\beta}\bar{z}_2(\bar\beta)) + \calS_{12}(\alpha,\bar\beta;\bar{z})\p_{\bar\beta}\bar{z}_1(\bar\beta) \right] \bar{z}_2(\bar\beta) d\bar\beta.
    \end{equation}
    Proceeding in a similar fashion, we also have
    \begin{equation}\label{eq:evenaux2}
    \p_t \bar{z}_2(\alpha) = -\frac{1}{4\pi}\int_\T \left[ \calS_{21}(\alpha,\bar\beta;\bar{z})(-\p_{\bar\beta}\bar{z}_2(\bar\beta)) + \calS_{22}(\alpha,\bar\beta;\bar{z})\p_{\bar\beta}\bar{z}_1(\bar\beta) \right] \bar{z}_2(\bar\beta) d\bar\beta,
    \end{equation}
    where
    $$
    \calS_{21}(\alpha,\beta;z) = - \frac{(z_2(\alpha)-z_2(\beta))\sin(z_1(\alpha) - z_1(\beta))}{\cosh(z_2(\alpha) - z_2(\beta)) - \cos(z_1(\alpha) - z_1(\beta))},
    $$
    \begin{align*}
    \calS_{22}(\alpha,\beta;z) = \log\left(2(\cosh(z_2(\alpha) - z_2(\beta))- \cos(z_1(\alpha) - z_1(\beta)))\right)\\
    - \frac{(z_2(\alpha) - z_2(\beta))\sinh(z_2(\alpha) - z_2(\beta))}{\cosh(z_2(\alpha) - z_2(\beta)) - \cos(z_1(\alpha) - z_1(\beta))}.
    \end{align*}
    From \eqref{eq:evenaux1}, \eqref{eq:evenaux2}, we conclude that $\bar{z}$ satisfies \eqref{eq:cde} for $\alpha \in (-\pi,0]$. An almost identical argument would verify the case when $\alpha \in (0,\pi]$. We omit the tedious details for simplicity.
\end{proof}

With the two symmetry properties established above, we introduce the following terminology:
\begin{defn}\label{defn:odd+even}
    Suppose $z(\alpha,t)$ is the solution to \eqref{eq:cde} with initial data $z_0 \in C^{1,\gamma}$ verifying $z(\alpha,t) = -z(-\alpha,t)$ and \eqref{evensym1}, then we say that $z(\alpha,t)$ satisfies \textbf{central+even symmetry}.
\end{defn}

The following corollary is immediate given Proposition \ref{prop:even}.
\begin{cor}\label{cor:positions}
    Suppose $z(\alpha,t)$ is the solution to \eqref{eq:cde} with initial data $z_0 \in C^{1,\gamma}$ possessing central+even symmetry. Then
    $$
    z_1(\pm\pi,t) = \pm\pi, \quad z_1\left(\pm\frac{\pi}{2},t\right) = \pm\frac{\pi}{2}, 
    $$
    $$
    z_2(\pm\pi,t) = 0, \quad \quad z_2\left(\frac{\pi}{2},t\right) = -z_2\left(-\frac{\pi}{2},t\right).
    $$
\end{cor}

\begin{rmk}
    As a straightforward consequence of Proposition \ref{prop:even}, if we assume that $z_0$ satisfies both central+even symmetry in the sense of \eqref{evensym1}, $\rho(x,t)$ necessarily satisfies the following symmetry:
    \begin{equation}
        \label{rhosym}
        \begin{split}
            \rho(-\pi-x_1,x_2,t) = \rho(x_1,x_2,t),\quad (x_1,x_2) \in (-\pi,0]\times \R,\\
             \rho(\pi-x_1,x_2,t) = \rho(x_1,x_2,t),\quad (x_1,x_2) \in (0,\pi]\times \R.
        \end{split}
    \end{equation}
    For the rest of this paper, we will interchangeably use \eqref{rhosym} and \eqref{evensym1}, whichever is more convenient.
\end{rmk}

\subsection{Evolution of the Potential Energy}
\label{sec_energy}
This section aims to study the following (modulated) potential energy, namely:
$$
E(t) := \int_{\T \times \R} x_2 (\rho_s(x_2) - \rho(x,t))dx,
$$
where we recall that
$$
\rho_s(x_2) = \begin{cases}
    1, & x_2 \ge 0,\\
    -1, & x_2 < 0,
\end{cases}
$$
is the unstably stratified state. 
A key observation is that $E(t)$ is monotone increasing in time:
\begin{lem}\label{lem:dtE}
    Assuming that $(\rho,u)$ is the solution to \eqref{eq:Stokes2mass}--\eqref{eq:Stokes2BS} with initial datum $\rho_0$, the following estimate holds:
    \begin{equation}\label{eq:potential}
    E'(t) = \|(-\Delta)^{-1}\p_1 \rho(t)\|_{L^2}^2 =: \delta(t).
    \end{equation}
\end{lem}
\begin{proof}
    Using formulation \eqref{eq:Stokes2mass}--\eqref{eq:Stokes2BS}, we compute that
    \begin{align*}
        E'(t) &= -\int_{\T \times \R}x_2 \p_t \rho(x,t) dx = \int_{\T \times \R} x_2 (u\cdot\nabla \rho)dx= -\int_{\T \times \R} u_2 \rho dx\\ &= -\int_{\T \times \R} \rho \p_1\Delta^{-2}\p_1 \rho dx= \int_{\T \times \R} \p_1\rho \Delta^{-2}\p_1 \rho dx= \|(-\Delta)^{-1}\p_1 \rho\|_{L^2}^2,
    \end{align*}
    where we also used that $(-\Delta)^{-1}$ is a symmetric operator.
\end{proof}

Moreover, the potential energy $E(t)$ can be used to estimate the heights of both the highest and lowest points of the interface $\Gamma(t)$ at any fixed time instance $t \ge 0$. To be more precise, let us define
$$
M(t) := \sup_{(x_1,x_2) \in \Gamma(t)}\{x_2\},\quad -m(t) := \inf_{(x_1,x_2) \in \Gamma(t)}\{x_2\}. 
$$
Then we have the following lemma:
\begin{lem}\label{lem:perimeterlb}
    Assuming that $(\rho,u)$ is the solution to \eqref{eq:Stokes2mass}--\eqref{eq:Stokes2BS} with initial datum $\rho_0$, the following lower bound holds:
    $$
    \max\{M(t), m(t)\} \ge \frac{1}{2\sqrt{\pi}}E(t)^{1/2}.
    $$
\end{lem}
\begin{proof}
    By definitions of $\rho_s$ and $\rho(x,t)$, we first observe that
    $$
    E(t) = 2\int_{\Omega^-(t) \cap \{x_2 > 0\}} x_2 dx - 2\int_{\Omega^{+}(t) \cap \{x_2 < 0\}} x_2 dx =: E^+(t) + E^{-}(t).
    $$
    Clearly, both $E^+$ and $E^-$ are nonnegative. Then let us first suppose that $E^+(t) \ge \frac12 E(t)$. By definition of $M(t)$, we know that $\Omega^- \subset \T \times (-\infty, M(t))$. Based on this observation, we have that
    \begin{align*}
        E^+(t) &\le 2\int_0^{M(t)}\int_\T x_2 dx_1 dx_2 = 2\pi M(t)^2,
    \end{align*}
    from which we obtain
    $$
    M(t) \ge \frac{1}{\sqrt{2\pi}}E^+(t)^{1/2} \ge \frac{1}{2\sqrt{\pi}}E(t)^{1/2}.
    $$
    Supposing otherwise, i.e. $E^-(t) \ge \frac12 E(t)$, a similar argument yields
    $$
    m(t) \ge \frac{1}{2\sqrt{\pi}}E(t)^{1/2}.
    $$
    Hence, the proof is completed.
\end{proof}
\begin{rmk}
    We remark that Lemma \ref{lem:dtE} and Lemma \ref{lem:perimeterlb} hold for solutions to \eqref{eq:cde} (equivalently \eqref{eq:Stokes2mass}--\eqref{eq:Stokes2BS}) without any symmetry assumptions.
\end{rmk}
 
\section{Large Growth in Some Geometric Quantities of the Interface}\label{sec:finiteE}
In this section, we prove Theorem \ref{thm:main}. Let us first outline the strategy of proving the main result.

\subsection{Outline of the Proof Strategy}
\label{sec:organization}
Consider the initial datum $z_0$ that assumes the central+even symmetry. We may further assume that $E(0) = 1$, where the general case would follow similarly by scaling. Our plan is to discuss a dichotomy on how fast the perimeter $L(t)$ grows. Let $\alpha \in (0,1), \beta > 0$ be parameters to be determined later. We split our proof into discussions of the following two cases:

\medskip
\noindent\textbf{Case 1: $\limsup_{t \to \infty} t^{-\alpha}L(t) > \beta$.} 

In this case, we immediately have
$$
\limsup_{t\to\infty}t^{-\alpha+\eps}L(t) = \infty
$$
for any $\eps > 0$.

\medskip
\noindent\textbf{Case 2: $\limsup_{t \to \infty} t^{-\alpha}L(t) \le \beta$.} 

In this case, we would have $L(t) \le 2\beta t^{\alpha}$ for all $t$ sufficiently large. From Lemma \ref{lem:perimeterlb} and the fact that $L(t) \ge M(t) + m(t)$, we immediately have $E(t) \le C(\beta)t^{2\alpha}$ for $t$ sufficiently large. Using \eqref{eq:potential}, one can extract an increasing sequence of times $t_n \nearrow \infty$ such that
$$
\delta(t_n) \le C(\alpha,\beta)t_n^{2\alpha - 1}.
$$
Hence in case where $\alpha < 1/2$ and $t \gg 1$, the heuristics above behooves us to consider the geometry of $\Gamma_t$ when $\delta(t) = \|(-\Delta)^{-1}\p_1 \rho\|_{L^2}^2 \ll 1$. As we will see in Section \ref{sec:finiteE}, the maximal curvature $\calK(t)$ will be bounded from below by some negative powers of $\delta(t)$. In particular, we will show that for sufficiently large $n$:
\begin{equation}\label{est:case2main}
\calK(t_n) \ge C(\alpha,\beta)\delta(t_n)^{-1/12} \ge C(\alpha,\beta) t_n^{\frac{1-2\alpha}{12}},
\end{equation}
after a careful choice of $\alpha$ and $\beta$. But \eqref{est:case2main} implies that
$$
\limsup_{t\to\infty} t^{\frac{2\alpha-1}{12} + \eps}\calK(t) =\infty
$$
for any $\eps > 0.$ The main result (Theorem \ref{thm:main}) follows from summarizing the above two cases after optimizing over the choices for $\alpha$ and $\beta$.

In the rest of this section, we will study Case 2 in detail. In Section \ref{subsect:prelim}, we first introduce some frequently used notations and establish two technical lemmas concerning $\delta(t)$ and some relevant geometric quantities. In Section \ref{subsect:case2}, we will use the preliminary results proved in Section \ref{subsect:prelim} to show two crucial geometric facts, which help justifying the estimate \eqref{est:case2main} above. Finally, we prove Theorem \ref{thm:main} according to the plan above.

\subsection{Notations and Technical Lemmas}\label{subsect:prelim}
Recalling the central+even symmetry assumptions stated in Definition \ref{defn:odd+even}, we simplify our analysis by only considering the dynamics in the restricted domain $D := (-\pi/2, \pi/2) \times \R$, and we define $D^\pm = \Omega^\pm \cap D$. We then denote the total length of $\Gamma_t \cap D$ by $\calL(t)$. It is clear by the symmetry assumption that
$$
\calL(t) = \int_{-\pi/2}^{\pi/2} |\dotz(\alpha)| d\alpha = \frac{L(t)}{2}.
$$
Again due to the symmetry assumptions, the maximal curvature of the curve $\Gamma_t \cap D$ coincides with $\calK(t)$. We thus slightly abuse the notation by still denoting $\calK(t)$ as the maximal curvature of the curve $\Gamma_t \cap D$. 

With the preparations above, we may start with deriving a quantitative estimate of $\delta(t)$.
\begin{lem}\label{lem:h-2}
    Suppose there exist an $\eps > 0$ and two disks $B_- = B_\epsilon (x^-) \subset D^-$, $B_+ = B_\eps (x^+) \subset D^+$, where $x_2^+ = x_2^-$. Then there exists a universal constant $C_0 > 0$ such that
    \begin{equation}
        \label{est:radiusbd}
        \eps \le C_0\delta(t)^{\frac15}.
    \end{equation}
\end{lem}
\begin{proof}
    Without loss of generality, assume $x_1^+ < x_1^-$. To begin with, we note that for an arbitrary test function $\varphi \in C^\infty_c (\T \times \R)$, the following computation holds:
    \begin{equation}\label{radbdaux1}
    \begin{split}
        \int_{\T \times \R} \p_1 \varphi \rho dx &= -\int_{\T \times \R} \varphi \p_1\rho dx = -\int_{\T \times \R} \varphi (-\Delta)(-\Delta)^{-1}\p_1 \rho dx\\
        &= -\int_{\T \times \R} (-\Delta \varphi)(-\Delta)^{-1}\p_1 \rho dx \le \|\Delta \varphi\|_{L^2}\delta(t)^{1/2}.
        \end{split}
    \end{equation}
    On the other hand, let us consider the test function $\varphi(x) = g(x_1)h(x_2)$, where
    \begin{align*}
    g(x_1) :=\begin{cases}
        1+ \sin\left(\frac{\pi}{2\eps}(x_1 - x_1^+)\right),& |x_1 - x_1^+| \le \eps,\\
        2,& x_1^+ + \eps < x_1 < x_1^- - \eps,\\
        1 - \sin\left(\frac{\pi}{2\eps}(x_1 - x_1^-)\right),& |x_1 - x_1^-| \le \eps,\\
        0,& \text{otherwise},
    \end{cases}
    \end{align*}
    and
    \begin{align*}
    h(x_2) = \begin{cases}
    1 + \cos\left(\frac{\pi}{2\eps}(x_2 - x_2^+)\right),& |x_2 - x_2^+| \le \eps,\\
    0,& \text{otherwise}.
    \end{cases}
    \end{align*}
    From the construction, we know that $$\supp(\nabla^2\varphi) \subset \{(x_1,x_2 )\in \T\times \R\;:\; x_1^+ - \eps\le x_1 \le x_1^- + \eps, x_2^+ - \eps \le x_2 \le x_2^+ +\eps\},$$ and $\|\Delta \varphi\|_{L^\infty} \le \eps^{-2}$. Thus an elementary computation yields
    \begin{equation}\label{radbdaux2}
    \|\Delta\varphi\|_{L^2} \le \|\Delta \varphi\|_{L^\infty}|\supp(\nabla^2\varphi)|^{1/2} \lesssim \eps^{-2}\cdot \eps^{\frac12} = \eps^{-\frac32}.
    \end{equation}
    Moreover, using the fact that $\rho \equiv \pm1$ in $B_\pm$ and $\supp(\p_1 \varphi) \subset B_+ \cup B_-$, we have
    \begin{equation}\label{radbdaux3}
    \int_{\T \times \R}\p_1\varphi \rho dx = \int_{B_+}\p_1 \varphi dx - \int_{B_-}\p_1\varphi dx.
    \end{equation}
    By the construction of $\varphi$, we note that
    \begin{equation}\label{h-2aux1}
    \begin{split}
        \int_{B_+} \p_1\varphi dx &= \frac{\pi}{2\eps}\int_{B_+}\cos\left(\frac{\pi}{2\eps}(x_1 - x_1^+)\right)\left(1 + \cos\left(\frac{\pi}{2\eps}(x_2 - x_2^+)\right)\right) dx\\
        &= \frac{2\eps}{\pi}\int_{B_{\pi/2}(0)}\cos(u_1)(1+\cos(u_2)) du\\
        &= c_0\eps,
    \end{split}
    \end{equation}
    where $c_0$ is a universal constant. A similar computation also yields
    \begin{equation}\label{h-2aux2}
    \int_{B_-}\p_1\varphi = -c_0\eps.
    \end{equation}
    Thus combining \eqref{h-2aux1}, \eqref{h-2aux2} with \eqref{radbdaux2}, we obtain
    $$
    \int_{\T\times\R}\p_1\varphi \rho dx = 2c_0\eps.
    $$
    Further combining with \eqref{radbdaux1} and \eqref{radbdaux2} yields the desired bound \eqref{est:radiusbd}.
\end{proof}

From now on, we define the closed tubular region $\Gamma_t^\eps := \overline{\cup_{x \in \Gamma_t} B_\eps (x) \cap D}$ and set $D^\pm_\eps(t) = D^\pm(t)\backslash \Gamma_t^\eps$; note that $D^\pm_\eps$ are both open sets by definition. Next, we define $h_\eps(t)$ as the lowest $x_2$-coordinate of the region $D^+_\eps$ in the lower-half plane, namely,
$$
h_\eps(t) := \inf\{a < 0:\; D^+_\eps \cap \{x_2 = a\} \neq \varnothing\}.
$$
See Figure~\ref{fig_eps} for an illustration of the sets $D_\epsilon^\pm$ and $\Gamma^\epsilon_t$, and the value $h_\epsilon(t)$.

\begin{figure}[htbp]
\begin{center}
\includegraphics[scale=1]{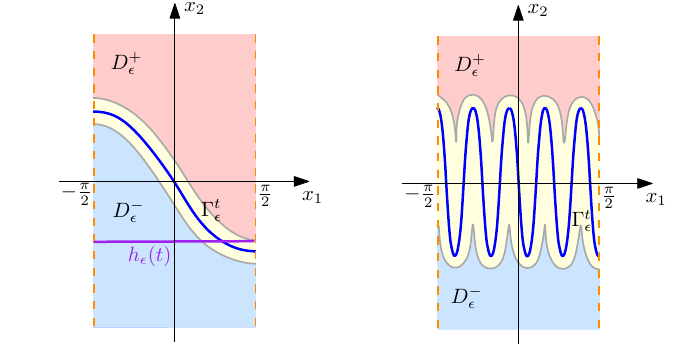}
\end{center}
\caption{\label{fig_eps}Illustration of the definitions of $D_\epsilon^+$ (red region), $D_\epsilon^-$ (blue region), and $\Gamma_\epsilon^+$ (yellow region), and $h_\epsilon(t)$ ($x_2$-coordinate of the purple line in the example on the left). Note that $h_\epsilon(t)$ might not be well-defined if $D_\epsilon^+ \cap \{x_2<0\}$ is empty, as illustrated in the example on the right.}
\end{figure}

However, we must be careful since the infimum contained in the definition of $h_\eps$ might not be well-defined. This is because \textit{a priori} it is possible that $D_\eps^+ \cap \{x_2 < 0\}$ is empty; see the right figure in Figure~\ref{fig_eps} for an example. However, we will show that $h_\eps$ is always well-defined as long as we choose $\eps$ to be appropriately small and the perimeter $\calL(t)$ is sufficiently bounded.

\begin{lem}\label{lem:h(t)}
    Assume that $\calL(t) < \frac{1}{2000}C_0^{-1/3}\delta(t)^{-\frac{1}{15}}$ and $\delta(t)$ is sufficiently small. Then $D_\eps^+ \cap \{x_2 < 0\}$ is nonempty for any $\eps \le {\frac52} C_0\delta(t)^{\frac15}$, where $C_0$ is the universal constant appearing in Lemma \ref{lem:h-2}. Moreover, we have
    \begin{equation}
        \label{est:Depslb}
        |D_\eps^+ \cap \{x_2 < 0\}| > 500C_0^{1/3}\delta(t)^{1/15}.
    \end{equation}
\end{lem}

\begin{proof}
    We prove this lemma by tracking the measure of the set $D_\eps^+ \cap \{x_2 < 0\}$. Due to our symmetry assumptions on the solution, it is straightforward to observe that
    $$
    E(t) = -4\int_{D^+ \cap \{x_2 < 0\}} x_2 dx = 4\int_{D^- \cap \{x_2 > 0\}} x_2 dx.
    $$
    Recalling the definition of $m(t)$, we have the following lower bound:
    \begin{equation}\label{welldefineaux1}
    |D^+ \cap \{x_2 < 0\}| \ge \frac{E(t)}{4m(t)}.
    \end{equation}
    On the other hand, we can straightforwardly estimate $\Gamma_t^\eps$ by $|\Gamma_t^\eps| \le 2\eps \calL(t)$. Combining this bound with \eqref{welldefineaux1} and recalling $D_\eps^+ = D^+\backslash \Gamma_t^\eps$, we have
    $$
    |D_\eps^+ \cap \{x_2 < 0\}| \ge \frac{E(t)}{4m(t)} - 2\eps \calL(t) \ge \frac{1}{2\calL(t)} - 2\eps\calL(t).
    $$
    Here we used $E(t) \ge E(0) = 1$ and $2m(t) \le \calL(t)$ in the second inequality above. By our assumptions on $\calL(t)$ and $\eps$, it is straightforward to verify that
    $$
    \eps \le {\frac52}C_0\delta(t)^{1/5} < \frac{\calL(t)^{-2}}{8}.
    $$
    given $\delta(t)$ sufficiently small. Combining the two inequalities above, we conclude that $$|D_\eps^+ \cap \{x_2 < 0\}| > \frac{1}{4\calL(t)} \ge 500C_0^{1/3}\delta(t)^{1/15}.$$
\end{proof}

\subsection{Proof of a Key Lemma}\label{subsect:case2}
In this section, we prove a key lemma which reveals that either the perimeter or the maximal curvature of the free interface $\Gamma_t$ has to increase when the quantity $\delta(t)$ decreases. This lemma is crucial in studying the Case 2 described in Section \ref{sec:organization}.
\begin{lem}[Key Lemma]
    \label{lem:case2main2}
    There exists a sufficiently small number $\delta_0 > 0$ only depending on the initial data, so that for $\delta(t) \le \delta_0$, at least one of the following statements holds:
    \begin{equation}\label{est:largeperi}
    \calL(t) \ge \frac{1}{2000}C_0^{-1/3}\delta(t)^{-1/15},
    \end{equation}
    or
    \begin{equation}\label{est:largecur}
    \calK(t) \ge \left({\frac{5}{2}}C_0\right)^{-1/3}\delta(t)^{-1/15},
    \end{equation}
    where $C_0$ is the universal constant appearing in Lemma \ref{lem:h-2}.
    
\end{lem}
The rest of this section will be dedicated to the proof of this lemma. To start with, we assume that
\begin{equation}
    \label{est:contradiction}
    \calL(t) < \frac{1}{2000}C_0^{-1/3}\delta(t)^{-1/15}.
\end{equation}
We select $\eps_0(t) := \frac{5}{2}C_0\delta(t)^{1/5}$, and consider $h(t) := h_{\eps_0(t)}(t)$, which is well-defined thanks to Lemma \ref{lem:h(t)}. By choosing $\delta_0$ sufficiently small and in view of $\delta(t) \le \delta_0$, we may also assume that $\eps_0 < \frac{1}{100}$. With the two lemmas in Section \ref{subsect:prelim}, we have the following observation concerning the horizontal line $\{x_2 = h(t)\}$.
\begin{cor}\label{cor:h}
    $\{x_2 = h(t)\} \subset \Gamma_t^{\eps_0(t)}$.
\end{cor}
\begin{proof}
    First by definition of $h(t)$, there exists $a \in (-\pi/2, \pi/2)$ such that $B_{\eps_0/2}(a,h(t)) \subset D^+$. Suppose for the sake of contradiction that $\{x_2 = h(t)\} \not\subset \Gamma_t^{\eps_0(t)}$. Then there exists $b \in (-\pi/2, \pi/2)$ such that $(b,h(t)) \in D^+_{\eps_0} \cup D^-_{\eps_0}$. In fact, we must further have $(b,h(t)) \in D^-_{\eps_0}$ due to the minimality of $h(t)$. This is because $D_{\eps_0}^\pm$ are open sets by definition. Hence, $B_{\eps_0/2}(b,h(t)) \subset D^-$. By Lemma \ref{lem:h-2}, however, we must have $\eps_0/2 \le C_0\delta(t)^{1/5}$, which contradicts the definition of $\eps_0$.
\end{proof}

Before we proceed, we remark on several elementary but important facts concerning the interface $\Gamma_t$, which relies heavily on the symmetry assumption. 

From \eqref{evensym1}, we know that $\Gamma_t \cap \bar D$ is a connected curve with two end points. This is because otherwise the whole curve $\Gamma_t$ would have a point of self-intersection, which contradicts the global well-posedness result stated in \cite{GGS25}. In view of Corollary \ref{cor:positions}, we have $\Gamma_t \cap \bar D = \{z(\alpha),\; \alpha \in [-\pi/2, \pi/2]\}.$ Moreover, $z_2(-\pi/2) = -z_2(\pi/2)$, and $z(-\pi/2), z(\pi/2)$ are the only two points on $\Gamma_t \cap \bar D$ which intersect $\p D$. Thus from now on, we may without loss assume that $z_2(-\pi/2) \ge 0$, as otherwise we could simply change the orientation of the curve by setting $\alpha \mapsto -\alpha$.

Moreover, let us consider the following collection of disks:
$$
B_i := B_{\eps_0}\left(-\frac{\pi}{2} + \eps_0(4i+1), h(t)\right),\quad i = 1,\hdots, \lfloor\frac{\eps_0^{-2/3}}{1600}\rfloor.
$$
By Corollary \ref{cor:h} and the definition of $\Gamma_t^{\eps_0}$, we know that there exists $\beta_i \in (-\pi/2,\pi/2)$ such that $z(\beta_i) \in \Gamma_t \cap \bar D \cap B_i$ for each $i$. 

We will show \eqref{est:largecur} by a contradiction argument. That is, we start by assuming that $\calK(t) < \eps_0^{-1/3} = \left(\frac{3}{2}C_0\right)^{-1/3}\delta(t)^{-1/15}$ and then demonstrate that \eqref{est:contradiction} is violated. The proof by contradiction fundamentally utilizes two geometric facts that are covered in the following two propositions. The first of which says that the tangent vector at each point $z(\beta_i)$ as defined above has to be almost vertical:
\begin{prop}
    \label{prop:largeslope}
    Assume that $\calK(t) < \eps_0^{-1/3}$ and the bound \eqref{est:contradiction} holds. For an arbitrary point $y = (y_1, y_2) = z(\beta) \in \Gamma_t \cap D$ with $d := y_1 + \pi/2 \le \frac{\eps_0^{1/3}}{400}$ and $y_2 \in [h(t) - \eps_0, h(t) + \eps_0]$, we have
    \begin{equation}
        \label{est:largeslope}
        \left|\frac{\dotz_2(\beta)}{\dotz_1(\beta)}\right| \ge \frac{\eps_0^{1/6}}{10d^{1/2}}.
    \end{equation}
\end{prop}
The second geometric fact is that for an arbitrary point $z(\beta) \in \Gamma_t \cap D$ whose tangent is not too horizontal, one must pay an arclength of at least order $\eps_0^{1/3}$ before the tangent vector becomes horizontal again.
\begin{prop}\label{prop:budget}
    Assume that $\calK(t) < \eps_0^{-1/3}$, and consider an arbitrary point $z(\beta) \in \Gamma_t\cap D$ with
    $$
    \left|\frac{\dotz_2(\beta)}{\dotz_1(\beta)}\right| \ge 1.
    $$
    Define
    $$
    \bar{\beta} := \inf\{\alpha > \beta\;|\; \dotz_2(\alpha) = 0\},\quad \underline{\beta} := \sup\{\alpha < \beta\;|\; \dotz_2(\alpha) = 0\}.
    $$
    Then
    $$
    \int_\beta^{\bar{\beta}} |\dotz(\alpha)| d\alpha \ge \frac{\eps_0^{1/3}}{2},\quad \int^\beta_{\underline{\beta}} |\dotz(\alpha)| d\alpha \ge \frac{\eps_0^{1/3}}{2}.
    $$
\end{prop}
\begin{rmk}
    The quantities $\bar\beta$ and $\underline{\beta}$ are indeed well-defined and $\bar\beta, \underline{\beta} \in [-\pi/2,\pi/2]$. This is because we must have $\dot{z}_2(-\pi/2) = \dot{z}_2(\pi/2) = 0$ due to our symmetry assumptions.
\end{rmk}
Before proving the aforementioned geometric facts, we demonstrate how to conclude the proof of Lemma \ref{lem:case2main2} given these two Propositions. First, we apply Proposition \ref{prop:largeslope} with $d = \eps_0^{1/3}/400$ to $z(\beta_i)$, which yields:
\begin{equation}\label{est:bislope}
\left|\frac{\dotz_2(\beta_i)}{\dotz_1(\beta_i)}\right| \ge \frac{\eps_0^{1/6}}{10}\frac{20}{\eps_0^{1/6}} = 2 > 1.
\end{equation}
Here, we used the fact that $|z_1(\beta_i) + \pi/2| \le \eps_0^{1/3}/400$ and $\eps_0 \ll 1$. Then applying Proposition \ref{prop:budget} to $z(\beta_i)$, we obtain that
$$
\int_{\beta_i}^{\bar{\beta_i}} |\dotz(\alpha)| d\alpha \ge \frac{\eps_0^{1/3}}{2},\quad \int^{\beta_i}_{\underline{\beta_i}} |\dotz(\alpha)| d\alpha \ge \frac{\eps_0^{1/3}}{2}.
$$
Now we consider the segments $\Gamma_i := \{z(\alpha)\;|\; \alpha \in (\underline{\beta_i}, \bar{\beta_i})\},$ and the following lemma holds:

\begin{lem}\label{lem:components}
   For $i \neq j$, $\Gamma_i, \Gamma_j$ are disjoint. 
\end{lem} 
\begin{proof}
    Given $i \neq j$, the lemma clearly follows from the following two facts:
    \begin{enumerate}
        \item $\Gamma_i, \Gamma_j$ either coincide or are disjoint;
        \item $z(\beta_j) \not\in \Gamma_i$.
    \end{enumerate}
    To prove the first fact, we note that if $z(\beta) \in \Gamma_i \cap \Gamma_j$ for some $\beta$, then $\bar\beta = \bar{\beta_i}$, $\underline{\beta} = \underline{\beta}_i$ if we view $z(\beta)$ as an element in $\Gamma_i$. An identical reasoning gives that $\bar\beta = \bar{\beta_j}$, $\underline{\beta} = \underline{\beta}_j$ if we see $z(\beta)$ as a member of $\Gamma_j$. But this implies that $\bar{\beta_i} = \bar{\beta_j}$, $\underline{\beta}_i = \underline{\beta}_j$, and therefore $\Gamma_i = \Gamma_j$.

    To prove the second fact, we first note that $\Gamma_i \cap \{x_2 = z_2(\beta_j)\} \neq \varnothing$ due to the fact that $|z_2(\beta_i) - h(t)| < \eps_0$ and Proposition \ref{prop:largeslope}. Define the (unique) point $z(\gamma) \in \Gamma_i$ such that $z_2(\gamma) = z_2(\beta_j)$. It is clear that it suffices just to show that $z(\gamma) \neq z(\beta_j)$, which amounts to showing that $z_1(\gamma) \neq z_1(\beta_j)$. To wit, we observe that for any $z(\beta) \in \Gamma_i$ such that $|z_2(\beta) - h(t)| < \eps_0$, we must have $\left|\frac{\dotz_2(\beta)}{\dotz_1(\beta)}\right| > \frac32$. One can prove this observation by viewing $\Gamma_i \cap \{x_2 \in (h(t) + \eps_0, h(t) - \eps_0)\}$ as a graph of $x_2$, and using the bound \eqref{est:bislope}, the curvature bound $\calK(t) < \eps_0^{-1/3}$, and $|z_2(\beta) - z_2(\beta_i)| < 2\eps_0$. We leave the details to interested readers. Hence, the above observation implies that $|z_1(\beta_i) - z_1(\gamma)| < \frac{4}{3}\eps_0$. On the other hand, we know that $|z_1(\beta_i) - z_1(\beta_j)| > 2\eps_0$ by construction of the collection of disks $\{B_k\}_k$. The two facts combined imply that $z_1(\gamma) \neq z_1(\beta_j)$, which concludes the proof of the lemma.
\end{proof}

Given Lemma \ref{lem:components}, we immediately have
$$
\calL(t) \ge \sum_{i = 1}^{\lfloor\frac{\eps_0^{-2/3}}{1600} \rfloor}|\Gamma_i| \ge \frac{\eps_0^{-1/3}}{1600} > \frac{1}{2000}C_0^{-1/3}\delta(t)^{-1/15},
$$
which contradicts \eqref{est:contradiction}. This concludes the proof of Lemma \ref{lem:case2main2}.

Now, we only need to prove Proposition \ref{prop:largeslope} and Proposition \ref{prop:budget}.

\subsubsection{Proof of Proposition \ref{prop:largeslope}}
The proof of the proposition hinges on the following geometric lemma.
\begin{lem}\label{lem:largeslopeaux1}
    Given the assumptions stated in Proposition \ref{prop:largeslope}, and additionally assuming that 
    \begin{equation}\label{est:lemlargeslopeaux1}
    \left|\frac{\dot{z_2}(\beta)}{\dot{z_1}(\beta)}\right| < \frac{\eps_0^{1/6}}{10d^{1/2}},
    \end{equation}
    we have
    $$
    \Gamma_t \cap \left\{\left(-\frac{\pi}{2}, y_2\right)\;|\; y_2 \in (h(t) - \eps_0^{1/3}, h(t) +\eps_0^{1/3})\right\} \neq \varnothing.
    $$
\end{lem}
Given Lemma \ref{lem:largeslopeaux1}, we prove Proposition \ref{prop:largeslope}. 

\begin{proof}[Proof of Proposition \ref{prop:largeslope}]
For the sake of contradiction, we assume that \eqref{est:largeslope} does not hold. On the one hand, we recall that $z(-\pi/2)$ with $z_2(-\pi/2) \ge 0$ is the only point on $\Gamma_t$ that intersects $\{x_1 = -\pi/2\}$ by the symmetry assumptions. This implies that $\Gamma_t$ cannot intersect the half-line $\{(-\pi/2, x_2),\; x_2 < 0\}$. 

On the other hand, since we also assume \eqref{est:contradiction}, we recall from \eqref{est:Depslb} that $|D_\eps^+ \cap \{x_2 < 0\}| > 500C_0^{1/3}\delta(t)^{{1/15}}$. Then by the definition of $h(t)$, we must have 
$$
h(t) + \eps_0^{1/3} \le -\frac{500}{\pi}C_0^{1/3}\delta(t)^{{1/15}} + \left({\frac52}C_0\right)^{1/3}\delta(t)^{{1/15}} < 0.
$$
However, this fact combined with Lemma \ref{lem:largeslopeaux1} implies that $\Gamma_t$ intersects the half-line $\{(-\pi/2, x_2),\; x_2 < 0\}$ given $\eps_0$ sufficiently small. This leads to a contradiction.
\end{proof}

We conclude this section by proving Lemma \ref{lem:largeslopeaux1}:
\begin{proof}[Proof of Lemma \ref{lem:largeslopeaux1}]
    To begin with, we additionally assume that $\left|\frac{\dot{z_2}(\beta)}{\dot{z_1}(\beta)}\right| \ge 1$. We will explain how to drop this extra assumption by the end of the proof. In this case, we may write $\Gamma_t$ as the graph $\{(f(x_2), x_2)\}$ locally around the point $z(\beta) = (y_1, y_2) \in \Gamma_t$ with $y_1 = f(y_2) = d -\pi/2$. Then the condition $1 \le \left|\frac{\dot{z_2}(\beta)}{\dot{z_1}(\beta)}\right| < \frac{\eps_0^{1/6}}{10d^{1/2}}$ is equivalent to
    \begin{equation}
        \label{est:largeslopeaux1-1}
        \frac{10d^{1/2}}{\eps_0^{1/6}} < |f'(y_2)| \le 1.
    \end{equation}
    We may without loss of generality assume that $f'(y_2) > 0$, as otherwise the argument follows in a similar fashion. Given the curvature bound $\calK(t) < \eps_0^{-1/3}$, we first observe that $f'(z)$ is uniformly bounded in an $\eps_0^{1/3}$ neighborhood of $y_2$. More precisely, we claim that for any $z \in (y_2 - \frac{\eps_0^{1/3}}{30}, y_2 + \frac{\eps_0^{1/3}}{30})$, 
    \begin{equation}\label{est:largeslopeaux1-2}
    |f'(z)| \le 2.
    \end{equation}
    We show the bound \eqref{est:largeslopeaux1-2} via a bootstrap argument: assuming \eqref{est:largeslopeaux1-2} and using the definition of scalar curvature, we have
    \begin{equation}\label{est:largeslopeaux1-3}
    |f''(z)| \le \calK(t)(1 + |f'(z)|^2)^{3/2} \le 15\eps_0^{-1/3},
    \end{equation}
    for any $z \in \left(y_2 - \frac{\eps_0^{1/3}}{30}, y_2 + \frac{\eps_0^{1/3}}{30}\right)$. Then for such $z$, we have
    \begin{align*}
        |f'(z)| &\le |f'(y_2)| + |f'(y_2) - f'(z)| \le 1 + \int_z^{y_2} |f''(\xi)|d\xi \le 1 + \frac{\eps_0^{1/3}}{30}\cdot 15\eps_0^{-1/3} \le \frac{3}{2} < 2,
    \end{align*}
    where we used \eqref{est:largeslopeaux1-1} and \eqref{est:largeslopeaux1-3} above. This closes the bootstrap argument. Note that an important consequence of \eqref{est:largeslopeaux1-2} is the uniform bound \eqref{est:largeslopeaux1-3} for the second derivative. For any $z \in (y_2 - \eps_0^{1/3}/30, y_2)$, we observe that
    \begin{align*}
        f'(z) \ge f'(y_2) - |f'(y_2) - f'(z)| \ge \frac{10d^{1/2}}{\eps_0^{1/3}} - \int_z^{y_2} |f''(\xi)| d\xi \ge \frac{10d^{1/2}}{\eps_0^{1/6}} -15\frac{y_2 - z}{\eps_0^{1/3}},
    \end{align*}
    where we used \eqref{est:largeslopeaux1-1} and \eqref{est:largeslopeaux1-3} above. Hence, we estimate $f(z)$ by:
    \begin{align*}
        f(z) &= f(y_2) - \int_z^{y_2} f'(\xi) d\xi \le d - \frac{\pi}{2}+\int_z^{y_2} \left(15\frac{y_2 - z}{\eps_0^{1/3}} - \frac{10d^{1/2}}{\eps_0^{1/6}}\right) d\xi\\
        &= \frac{15}{2\eps_0^{1/3}}(y_2 - z)^2 - \frac{10d^{1/2}}{\eps_0^{1/6}}(y_2 - z) + d - \frac{\pi}{2}.
    \end{align*}
    Upon inspecting the quadratic polynomial $p(X) = \frac{15}{2\eps_0^{1/3}}X^2 - \frac{10d^{1/2}}{\eps_0^{1/6}}X + d$, we know that $p(\frac{2}{3}d^{1/2}\eps_0^{1/6}) < 0$. Since $d \le \frac{1}{400}\eps_0^{1/3}$, we have $\frac{2}{3}d^{1/2}\eps_0^{1/6} \le \frac{\eps_0^{1/3}}{30}$ and indeed $f(y_2 - \frac{2}{3}d^{1/2}\eps_0^{1/6}) \le -\frac{\pi}{2}$. This implies that there exists $\overline{y_2} \in [y_2 - \frac{2}{3}d^{1/2}\eps_0^{1/6}, y_2]$ such that $f(\overline{y_2}) = -\frac{\pi}{2}$. Finally, since $|y_2 - h(t)| \le \eps_0$ by assumption, 
    \[
    |\overline{y_2} - h(t)| \le |\overline{y_2} - y_2| + |y_2 - h(t)| \le \frac{\eps_0^{1/3}}{30} + \eps_0 < \frac{\eps_0^{1/3}}{15},
    \]
    given $\eps_0 > 0$ sufficiently small. We have thus proved the desired result.

    Finally, we explain how to lift the extra assumption that $\left|\frac{\dot{z_2}(\beta)}{\dot{z_1}(\beta)}\right| \ge 1$. Suppose $\left|\frac{\dot{z_2}(\beta)}{\dot{z_1}(\beta)}\right| < 1$ and write $(y_1, y_2) = (z_1(\beta), z_2(\beta))$. We may locally write $z(\beta)$ as a graph of $y_2 = g(y_1)$. Then an identical argument to that showing \eqref{est:largeslopeaux1-2} would yield $g'(y') \le 2$ for any $y' \in (y_2 - \frac{\eps_0^{1/3}}{30}, y_2 + \frac{\eps_0^{1/3}}{30})$. Using $d<\frac{\eps_0^{1/3}}{400}$, the graph of $g$ can at most go up a distance $2d$ before hitting the vertical line $y_1=-\frac{\pi}{2}$, thus concluding the proof.
\end{proof}

A picture demonstrating the idea of Lemma \ref{lem:largeslopeaux1} can be seen in Figure \ref{fig:lemma1}, which depicts the ``worst case scenario'':
\begin{figure}[h]
    \centering
    \includegraphics[width=0.5\textwidth]{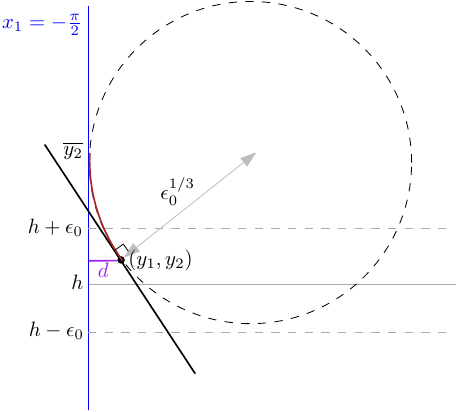}
    \caption{Illustration of Proof of Proposition \ref{prop:largeslope}. In the worst case scenario, the tangent line of the curve $z$ at $(y_1,y_2)$ is such that the tangent circle with radius $\epsilon_0^{1/3}$ is also tangent to the vertical line $x_1=-\frac{\pi}{2}$.}
    \label{fig:lemma1}
\end{figure}

\subsubsection{Proof of Proposition \ref{prop:budget}}
Finally, we prove the geometric fact given by Proposition \ref{prop:budget}. Below we provide a picture giving the worst case scenario (see Figure \ref{fig:lemma2}).

\begin{figure}[h]
    \centering
    \includegraphics[width=0.5\textwidth]{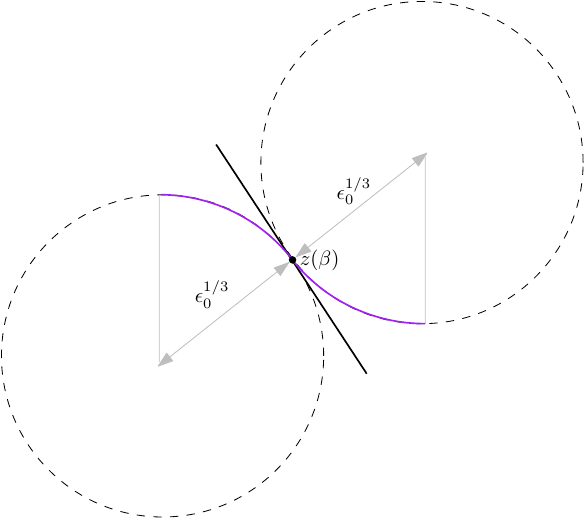}
    \caption{Illustration of the worst case scenario for Proposition \ref{prop:budget}. When the curve  $z$ coincides with the purple curve, it takes arclength $\sim \epsilon_0^{1/3}$ from the point $z(\beta)$ for the tangent to be horizontal.}
    \label{fig:lemma2}
\end{figure}

\begin{proof}[Proof of Proposition \ref{prop:budget}]
It suffices to prove the statement corresponding to $\bar{\beta}$, as the case corresponding to $\underline{\beta}$ follows similarly. If $\dot{z_1}(\alpha) \neq 0$, we define $\theta(\alpha)$ by
$$
\tan(\theta(\alpha)) = \frac{\dot{z_2}(\alpha)}{\dot{z_1}(\alpha)}.
$$
Observe that $\tan(\theta(\bar\beta)) = 0$ and
\begin{equation}\label{est:budgetaux1}
|\theta'(\alpha)| = \left|\frac{\dot{z_1}(\alpha)\ddot{z_2}(\alpha) - \ddot{z_1}(\alpha)\dot{z_2}(\alpha)}{(\dot{z_1}(\alpha))^2 + (\dot{z_2}(\alpha))^2}\right| = \frac{\left|\dot{z_1}(\alpha)\ddot{z_2}(\alpha) - \ddot{z_1}(\alpha)\dot{z_2}(\alpha)\right|}{|\dot{z}(\alpha)|^3}|\dot{z}(\alpha)| = \kappa(\alpha)|\dot{z}(\alpha)|,
\end{equation}
where $\kappa$ denotes the (unsigned) scalar curvature.

Now using the assumption of the Proposition and the fact that $\Gamma_t$ is a $C^2$ curve, there exists $\tilde{\beta} \in [\beta, \bar{\beta})$ such that $1 = |\tan(\theta(\tilde\beta))| > |\tan(\theta(\gamma))|$ for any $\gamma \in (\tilde\beta, \bar{\beta}]$. Note that from our choice of $\tilde\beta$, $|\dot{z_1}(\gamma)| > 0$ for any $\gamma \in (\tilde\beta, \bar{\beta}]$. Moreover, we have the following improved lower bound for $|\dot{z_1}(\gamma)|$:
\begin{equation}\label{est:budgetaux2}
|\dot{z}(\gamma)| = |\dot{z_1}(\gamma)|\sqrt{1 + \tan^2(\theta(\gamma))} < \sqrt{2}|\dot{z_1}(\gamma)|.
\end{equation}
Using \eqref{est:budgetaux1}, we have that
\begin{align*}
    1 &= \left|\tan(\theta(\tilde\beta)) - \tan(\theta(\bar\beta))\right| = \left|\int_{\tilde\beta}^{\bar\beta} \sec^2(\theta(\gamma))\theta'(\gamma) d\gamma\right|\\
    &\le \int_{\tilde\beta}^{\bar\beta} \frac{|\dot{z}(\gamma)|^2}{|\dot{z_1}(\gamma)|^2}|\kappa(\gamma)||\dot{z}(\gamma)|d\gamma\\
    &\le 2\eps_0^{-1/3}\int_{\tilde\beta}^{\bar\beta}|\dot{z}(\gamma)|d\gamma,
\end{align*}
where we used \eqref{est:budgetaux2} and the assumption $\calK(t) \le \eps_0^{-1/3}$ in the final inequality. We thus conclude the proof after rearranging the inequality above and using the definition of $\tilde\beta$.
\end{proof}

\subsection{Proof of Theorem \ref{thm:main}}
We conclude by proving the main result, Theorem \ref{thm:main}.
\begin{proof}[Proof of Theorem \ref{thm:main}]
    Let $\eta$ be any positive number which satisfies the following inequality:
    $$
    \eta^{{17/15}} < \frac{1}{2000}C_0^{-1/3}\left(\frac{17}{4\pi}\right)^{1/15},
    $$
    where we recall that $C_0$ is the constant defined in Lemma \ref{lem:h-2}.
    Following the plan in Section \ref{sec:organization}, we discuss the following two cases:
    
    \noindent\textbf{Case 1: $\limsup_{t\to\infty}t^{-1/17}L(t) \ge \eta$.} 
    
    In this case, we observe that for arbitrary $\eps > 0$,
    \begin{equation}
        \label{est:mainaux1}
        \limsup_{t\to\infty}t^{-1/17 +\eps} L(t) = \infty.
    \end{equation}
    Indeed, supposing otherwise, we have
    $$
   \eta \le \limsup_{t\to\infty}t^{-1/17}L(t) \le \left(\limsup_{t\to\infty} t^{-1/17+\eps}L(t)\right)\left(\limsup_{t\to\infty} t^{-\eps}\right) = 0,
    $$
    which is a contradiction.

    \noindent\textbf{Case 2: $\limsup_{t\to\infty}t^{-1/17}L(t) < \eta$.} 

    In this case, there exists $T > 0$ sufficiently large such that for all $t \ge T$, 
    \begin{equation}\label{est:case2Lbound}
    \calL(t) = \frac{L(t)}{2} < \eta t^{1/17}.
    \end{equation}
    Combining this inequality with Lemma \ref{lem:perimeterlb} and the fact that $\calL(t) \ge M(t) + m(t)$, we have
    \begin{equation}\label{est:case2Ebound}
    E(t) \le \pi\eta^2 t^{2/17}.
    \end{equation}
    In view of \eqref{eq:potential}, there is a sequence of times $t_n \nearrow \infty$, $t_n \ge T$, such that
    \begin{equation}\label{est:case2deltabound}
    \delta(t_n) \le \pi\eta^2\cdot \frac{4}{17}t_n^{-15/17}.
    \end{equation}
    Indeed, suppose the statement above were not true. Then for any $t \ge T$, we have
    \begin{align*}
        E(t) &= E(T) + \int_T^t \delta(s) ds \ge E(T) + 2\pi\eta^2 (t-T)^{2/17} \ge 2\pi\eta^2 (t-T)^{2/17},
    \end{align*}
    where we recall that $E(T) \ge E(0) \ge 0$. By choosing $t$ sufficiently large, we would have $E(t) \ge 2\pi\eta^2 (t-T)^{2/17} > \pi\eta^2 t^{2/17}$, which contradicts \eqref{est:case2Ebound}.
    
    Given \eqref{est:case2deltabound}, we claim that the following inequality must hold for all $t_n$:
    \begin{equation}
        \label{est:case2Lbound2}
        \calL(t_n) < \frac{1}{2000}C_0^{-1/3}\delta(t_n)^{-1/15}.
    \end{equation}
    If not, then applying \eqref{est:case2deltabound} we have
    \begin{align*}
        \calL(t_n) \ge \frac{1}{2000}C_0^{-1/3}\left(\pi\eta^2\cdot \frac{4}{17}t_n^{-15/17}\right)^{-1/15} > \eta t_n^{1/17},
    \end{align*}
    where we used the definition of $\eta$ in the final inequality. However, since $t_n \ge T$, the inequality above contradicts \eqref{est:case2Lbound}. Then by the Key Lemma \ref{lem:case2main2}, we conclude that
    \begin{equation}\label{est:mainaux2}
    \calK(t_n) \ge \left({\frac{5}{2}}C_0\right)^{-1/3}\delta(t_n)^{-1/15} \ge c_1 t_n^{1/17},
    \end{equation}
    where $c_1$ is a universal constant. Here, we also used \eqref{est:case2deltabound}. However, \eqref{est:mainaux2} immediately implies that for arbitrary $\eps > 0$,
    \begin{equation}
        \label{est:mainaux3}
        \limsup_{t\to\infty}t^{-1/17+\eps}\calK(t) = \infty.
    \end{equation}
    The proof is thus completed after we combine \eqref{est:mainaux1} and \eqref{est:mainaux3}.
\end{proof}

\section{Estimate on the Evolution of the Number of Fingers}\label{sec:relaxation}
In this section we  give an estimate on the number of fingers that could be developed by the flow. In order to do that we recall the following result (see \cite{grujic2000spatial} and the references therein)
\begin{lem}\label{grujic}
Let $L$, $\tau>0$, and let $u$ be analytic in the neighborhood of $\{\alpha+i\beta: |\beta|\leq \tau\}$ and $L$-periodic in the $\alpha$-direction. Then, for any $\mu>0$, $[0, L]=I_\mu\cup R_\mu$, where $I_\mu$ is an union of at most $[\frac{2L}{\tau}]$ intervals open in $[0, L]$, and
\begin{itemize}
\item $|\partial_\alpha u(\alpha)| \leq \mu, \text{ for all }\alpha\in I_\mu,$
\item $\text{Card}\{\alpha \in R_\mu : \partial_\alpha u(\alpha)=0\}\leq
\frac{2}{\log 2}\frac{L}{\tau}\log\left(\frac{\max_{|\beta|\leq \tau}|\partial_\alpha u(\alpha+i\beta)|}{\mu}\right).$
\end{itemize}
\end{lem}
\begin{proof}[Proof of Theorem \ref{thm:2}]
Let us introduce the following space of analytic functions
\begin{equation}\label{wiener}
A^{s}_\nu=\{u\in L^2(\mathbb{T}): \sum_{k=-\infty}^\infty e^{\nu |k|}|k|^{s}|\hat{u}(k)|<\infty\},\quad\mbox{where}\quad \hat{u}(k)=\frac1{2\pi}\int_\T u(\alpha)e^{-ik\alpha}d\alpha.
\end{equation}
In our case, fixing $\nu_0$ and invoking Theorems 2 and 3 in \cite{GGS25_SIMA}, we can ensure the existence of initial data $h_0(\alpha)$ such that
\begin{enumerate}
\item the corresponding solutions are analytic in a strip of width $\nu^*\leq \nu_0/24$,
\item the corresponding solutions satisfy the bound
$$
\max_{|\beta|\leq \nu^*}|\partial_\alpha h(\alpha+i\beta,t)|\leq \|h_0\|_{A^1_{\nu_0}}.
$$
\end{enumerate}
If we now fix $T$ and define the functions
$$
g(\alpha,t)=h(\alpha,T-t),
$$
we have that $g(\alpha,t)$ satisfy the RT unstable Stokes-Transport system. Recalling the ideas of Theorem 4 in \cite{GGS25_SIMA}, we have that
$$
\|g(T)\|_{A^1_{\nu^*}}\geq \|g(T)\|_{A^0_{\nu^*}}\geq C_1e^{C_2 \sqrt{T}}\|g_0\|_{A^0},
$$
where $C_1, C_2$ are harmless constants only depending on the parameters $\nu^\ast, \rho^+, \rho^-$ and the initial data. Invoking the previous lemma we can ensure the desired result: the spatial domain can be split as $[-\pi, \pi]=I_\mu\cup R_\mu$, where $I_\mu$ is an union of at most $[\frac{4\pi}{\nu^*}]$ open intervals and
\begin{itemize}
\item $|\partial_\alpha g(\alpha,t)| \leq \mu, \text{ for all }\alpha\in I_\mu,$
\item $\text{Card}\{\alpha \in R_\mu : \partial_\alpha g(\alpha,T)=0\}\leq
F(T),$
\end{itemize}
where the growth of $F(t)$ is bounded below as follows
$$
C(1+\sqrt{T})\leq \frac{2}{\log 2}\frac{2\pi}{\nu^*}\log\left(\frac{C_1e^{C_2 \sqrt{T}}\|g_0\|_{A^0}}{\mu}\right)\leq F(T).
$$
\end{proof}

\section{Rayleigh-Taylor Breakdown}\label{sec:RTbreakdown}

In this section, we prove the third main result of this article, namely Theorem \ref{thm:turning}.

\begin{proof}
    We solve the CDE for some time length $T>0$ in $[-T,T]$. It was proved in \cite{GGS25} that the system is globally well-posed and time reversible for smooth initial data. We denote by $z(\alpha,t)$ the free interface resulting of the evolution. Now, let us define 
    \begin{equation*}
        l(t) := \min_\alpha \partial_\alpha z_1(\alpha,t).
    \end{equation*}
    Then, at time $t=0$, the minimum is attained only at $\alpha =0$ because of condition 2, and $l(0) = 0$. Moreover, the following dichotomy holds: 
    \begin{itemize}
        \item If $\partial_\alpha u_1(0,0) = \partial_\alpha \partial_t {z}_1(0,0) >0$, then, there is a $\varepsilon >0$ such that for the time length $t < \varepsilon < T$, it holds $l(t) >0$, and therefore $\partial_\alpha z_1(\alpha,t) >0$ and the curve can be parametrized by a graph. 
        \item If  $\partial_\alpha u_1(0,0) = \partial_\alpha \partial_t {z}_1(0,0) <0$, then for a short time, $l(t) <0$ and the curve fails to be parametrized as a graph.
    \end{itemize}
    Therefore, the condition $\partial_\alpha u_1(0,0) = \partial_\alpha \partial_t {z}_1(0,0) <0$ ensures that there exists a time length $|t|<\varepsilon$ where: $z(\alpha,t)$ is parametrized by a graph in $t \in (-\varepsilon,0)$, has a vertical tangent at $t = 0$, and cannot be parametrized by a graph in $t \in (0,\varepsilon)$.
\end{proof}

All that is missing is to prove that there exist initial data $\bar{z}$ meeting the conditions of the theorem above. In order to prove that, we will first derive a simplified Contour Dynamics Equation that is better suited to verify condition 3 of the theorem. 

\begin{prop}
    Let us consider the Stokes-Transport system in its Contour Dynamics form \eqref{eq:cde}. Then,  
    \begin{equation} \label{eq:cde2}
        \partial_t z(\alpha,t) = (\rho^--\rho^+) \int_\mathbb{T} \partial_1 K(z(\alpha,t)-z(\beta,t))  \dot z(\beta,t) d\beta, \quad \alpha \in \mathbb{T}
    \end{equation}
    is an equivalent Contour Dynamics Equation, where $K(x)$ is the fundamental solution to the bilaplacian equation, that is 
    $$ \Delta^2 K(x) = \delta(x), \quad x \in \mathbb{T} \times \mathbb{R}.$$
\end{prop}
\begin{proof}
    We recall the general expression of the velocity field given in \eqref{eq:Stokes2BS} and make the following alternative distribution of the derivatives (in a weak sense):
    \begin{equation} 
    \label{eq:Stokes3}
         u = \nabla^\perp \Delta^{-2}\p_1 \rho = \p_1 K \ast \nabla^\perp \rho.
    \end{equation}
    It is not difficult to see that the weak gradient of $\rho(x,t)$ is given by
    \begin{equation} \label{eq:gradden}
        \nabla \rho(x,t) = - (\rho^--\rho^+)  \dot z^\perp(\alpha,t)\delta (x-z(\alpha,t)).
    \end{equation}
    Joining \eqref{eq:Stokes3} and \eqref{eq:gradden} together with the regularity of $\partial_1K$ provide the continuity of the velocity field across the interface and therefore \eqref{eq:cde2}.
\end{proof}

\begin{rmk}
    The kernel $\p_1 K$ is not available as an explicit combination of standard functions. Instead, it is given in \cite{GGS25_SIMA} as the infinite sum 
    \begin{equation*}
        \partial_1 K(x) = -\frac{1}{4\pi} \sum_{n= 1}^\infty  \frac{n|x_2|+1}{n^2} e^{-n|x_2|} \sin(nx_1).
    \end{equation*}
    However, the second derivatives of $K$ have an explicit expression (they form the Stokeslet) and we will only need 
    \begin{equation*}
        \p_1 \p_2 K(x) = \frac{1}{8\pi}\frac{x_2 \sin(x_1)}{\cosh(x_2)-\cos(x_1)}
    \end{equation*}
    for our result.
\end{rmk}

The final result of this section shows the existence of smooth initial data that evolves into a turning instability. 

\begin{prop} \label{prop:turning1}
    There exists a family of curves that meets the conditions 1-3 in Theorem \ref{thm:turning}.
\end{prop}
\begin{proof}
We prove that the family of curves constructed in  \cite{castro2012rayleigh} fulfills conditions 1-3. We recall the construction here for completeness. The first component of the curve is defined as 
     $$\bar{z}_1(\alpha) = \alpha-\sin(\alpha), \,\, \alpha \in (-\pi,\pi]$$
     and the second component as 
    \[
\bar{z}_2(\alpha) =
\begin{cases}
b z^*(\alpha), & \text{if } 0\leq \alpha \leq \alpha_2, \\
z^*(\alpha),  & \text{if } \alpha_2 < \alpha \leq \pi,
\end{cases}
\]
where $\bar{z}_2(\alpha)$ is extended to the negative interval by odd symmetry. Here, $b>0$ is to be fixed later, $0 < \alpha_2 < \alpha_3 < \frac{\pi}{2}$ and  $z^\ast(\alpha)$ is an odd smooth function defined in $[-\pi,\pi]$ with the properties:
\begin{enumerate}[(a)]
    \item \( \left( \partial_\alpha z^* \right)(0) > 0 \),
    \item \( z^*(\alpha) > 0 \) if \( \alpha \in (0, \alpha_2) \), 
    \item \( z^*(\alpha) < 0 \) if \( \alpha \in [\alpha_2, \alpha_3) \), 
    \item \( z^*(\alpha) \leq 0 \) if \( \alpha \in [\alpha_3, \pi] \).
\end{enumerate}
 With this definition, it is clear that $\bar{z}$ fulfills conditions 1-2. It remains to show condition 3.  We set $\rho^--\rho^+ = 1$ for simplicity, as we want to show the turning of initially stable interfaces. It holds that 
\begin{equation*}
    \partial_t \partial_\alpha  z(\alpha,t) = \int_\mathbb{T} \nabla \p_1 K(  z(\alpha) - {z}(\beta)) \cdot \partial_\alpha z(\alpha) \partial_\alpha  z(\beta) d\beta .
\end{equation*}
In particular, 
\begin{align}
     \partial_t \partial_\alpha \bar z_1(0,0) &= \int_\mathbb{T} \nabla \p_1 K( \bar z(0) - \bar{z}(\beta)) \cdot \partial_\alpha \bar z(0) \partial_\alpha \bar z_1(\beta) d\beta \nonumber  \\
     & = \int_\mathbb{T} \p_2 \p_1 K(  - \bar{z}(\beta)) \partial_\alpha \bar z_2(0) \partial_\alpha \bar z_1(\beta) d\beta \nonumber  \\
     & =\partial_\alpha \bar z_2(0) \frac{1}{8\pi}  \int_\mathbb{T} \frac{\bar z_2(\beta) \sin(\bar z_1(\beta))}{\cosh(\bar z_2(\beta))-\cos(\bar z_1(\beta))} \partial_\alpha \bar z_1(\beta) d\beta \nonumber \\
     & = \partial_\alpha \bar z_2(0) \frac{1}{4\pi}  \int_{0}^\pi \frac{\bar z_2(\beta) \sin(\bar z_1(\beta))}{\cosh(\bar z_2(\beta))-\cos(\bar z_1(\beta))} \partial_\alpha \bar z_1(\beta) d\beta, \label{eq:turningint}
\end{align}
where the last step is due to the odd symmetry of $\bar{z}$. Now, we split 
\begin{equation*}
     \partial_t \partial_\alpha \bar{z}_1(0,0) = J_1 + J_2,
\end{equation*}
where
\begin{align*}
    J_1 &= \partial_\alpha \bar z_2(0) \frac{1}{4\pi}  \int_{0}^{\alpha_2}  \frac{\bar z_2(\beta) \sin(\bar z_1(\beta))}{\cosh(\bar z_2(\beta))-\cos(\bar z_1(\beta))} \partial_\alpha \bar z_1(\beta) d\beta 
\end{align*}
and 
\begin{align*}
    J_2 &= \partial_\alpha \bar z_2(0) \frac{1}{4\pi}  \int_{\alpha_2}^{{\pi}} \frac{\bar z_2(\beta) \sin(\bar z_1(\beta))}{\cosh(\bar z_2(\beta))-\cos(\bar z_1(\beta))} \partial_\alpha \bar z_1(\beta) d\beta. 
\end{align*}
Note that due to the hypotheses on $\bar{z}$, all the terms inside the integrals are non-negative except for $\bar z_2$. On the one hand, 
\begin{equation}
\lim_{b \to \infty} J_1 = \lim_{b \to \infty} \partial_\alpha \bar z_2(0) \frac{1}{4\pi}  \int_{0}^{\alpha_2}  \frac{\bar z_2(\beta) \sin(\bar z_1(\beta))}{\cosh(\bar z_2(\beta))-\cos(\bar z_1(\beta))} \partial_\alpha \bar z_1(\beta) d\beta  =0,
\end{equation}
as the leading growth in $b$ is in the denominator.
On the other hand, 
\begin{equation}
  J_2 = \partial_\alpha \bar z_2(0) \frac{1}{4\pi}  \int_{\alpha_2}^{{\pi}} \frac{\bar z_2(\beta) \sin(\bar z_1(\beta))}{\cosh(\bar z_2(\beta))-\cos(\bar z_1(\beta))} \partial_\alpha \bar z_1(\beta) d\beta < 0
\end{equation}
independently of $b$. Therefore, if we fix $b$ large enough, we have that the sum of the two integrals is negative and condition 3 is fulfilled. 
\end{proof}

\begin{prop}
    There exists a family of curves that meets the conditions 1-4 in Theorem \ref{thm:turning}.
\end{prop}
\begin{proof}
    We adapt the family of curves constructed in the previous result to guarantee that the symmetry in condition 4 holds.

   There are many ways of defining $\bar{z}_1$. We give one possible construction. We define $\bar{z}_1$  as the piecewise function 
     \[
\bar{z}_1(\alpha) =
\begin{cases}
\alpha-\sin(\alpha), & \text{if } 0\leq \alpha \leq \alpha_1, \\
f(\alpha) & \text{if } \alpha_1 < \alpha \leq \frac{\pi}{2} \\
\end{cases}
\]
where $f(\alpha)$ is a smooth transition, positive and increasing. Furthermore, we extend $\bar{z}_1(\alpha)$ by imposing the symmetries
$$\bar{z}_1(\alpha) = \pi - \bar{z}_1(\pi-\alpha), \quad \bar{z}_1(\alpha) = - \bar{z}_1 (-\alpha).$$ We consider a  transition $f$ that ensures smoothness in the connecting point $\alpha = \frac{\pi}{2}$. It holds that $\partial_\alpha \bar{z}_1(\alpha) >0$ for all $\alpha \neq 0$ and $\partial_1 \bar{z}(0) = 0.$ 

\begin{figure}[h]
    \centering
\includegraphics[width=0.5\textwidth]{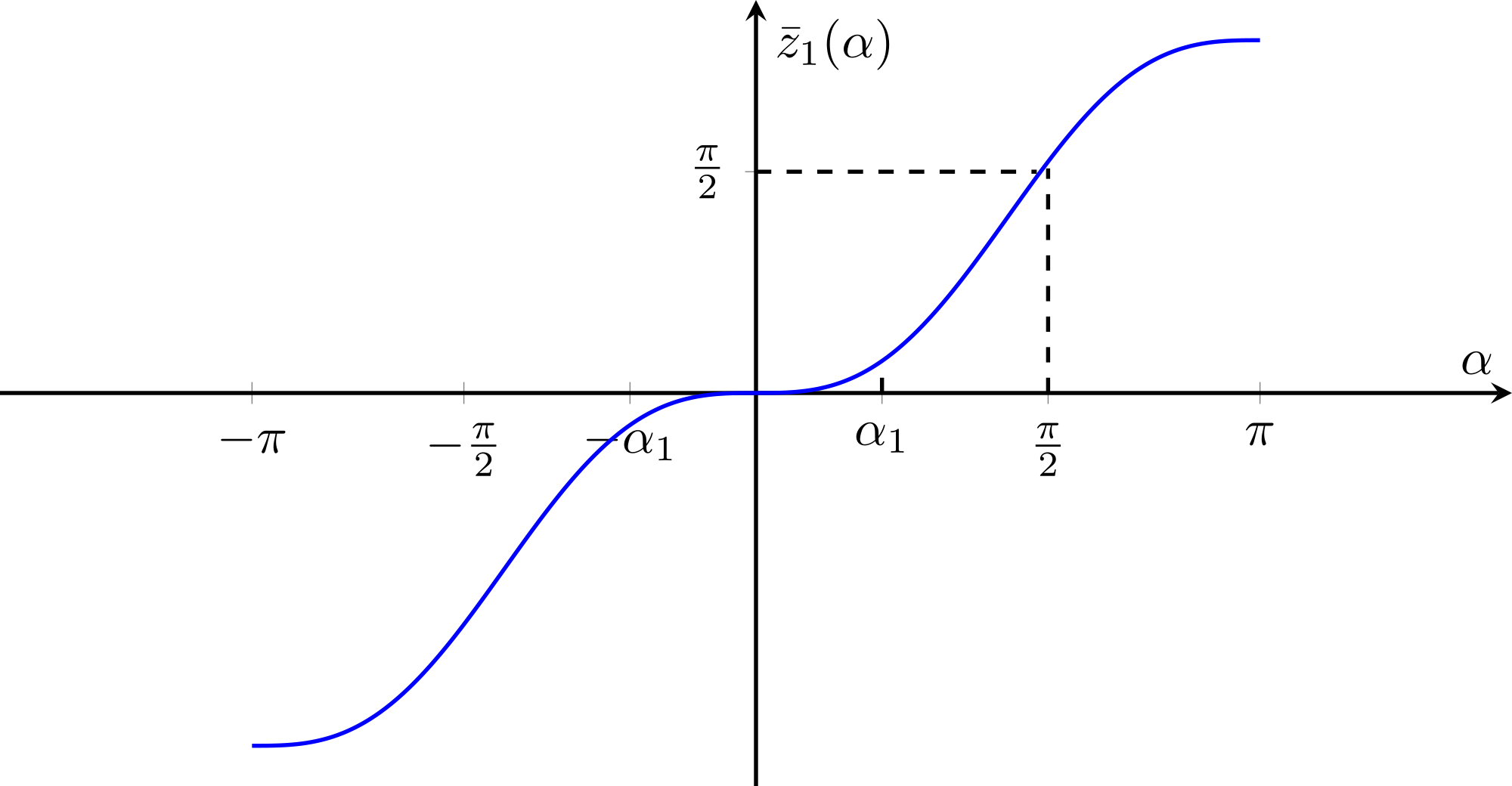}
  \caption{Sketch of a possible $\bar{z}_1(\alpha)$ fulfilling conditions 1-4.}
\end{figure}

Next, we define $\bar{z}_2(\alpha)$ very similarly as in Proposition \ref{prop:turning1}, adapting the construction to the new symmetry in condition 4: 
    \[
\bar{z}_2(\alpha) =
\begin{cases}
b z^*(\alpha), & \text{if } 0\leq \alpha \leq \alpha_2, \\
z^*(\alpha),  & \text{if } \alpha_2 < \alpha \leq \frac{\pi}{2},
\end{cases}
\]
where $\bar{z}_2(\alpha)$ is extended to $\frac{\pi}{2} < \alpha \leq \pi$ by even symmetry with respect to $\frac{\pi}{2}$ to fulfill \eqref{evensym} and it is extended to the negative interval by odd symmetry. As before, $b>0$ is sufficiently large and $0 < \alpha_2 < \frac{\pi}{2}$. Additionally, $\varepsilon >0$ is a small parameter and $z^\ast(\alpha)$ is an odd smooth function defined in $[-\frac{\pi}{2},\frac{\pi}{2}]$ with the properties:
\begin{enumerate}[(a)]
    \item \( \left( \partial_\alpha z^* \right)(0) > 0 \),
    \item \( z^*(\alpha) > 0 \) if \( \alpha \in (0, \alpha_2-\varepsilon) \), 
    \item \( z^*(\alpha) < 0 \) if \( \alpha \in (\alpha_2+\varepsilon, \frac{\pi}{2}-\varepsilon) \), 
    \item $z^\ast(\alpha) = 0$ otherwise.
\end{enumerate}
Notice that the flat transitions ensure smoothness in the connecting points $\alpha = \pm \frac{\pi}{2}$. 
\begin{figure}[h]
    \centering
\includegraphics[width=0.5\textwidth]{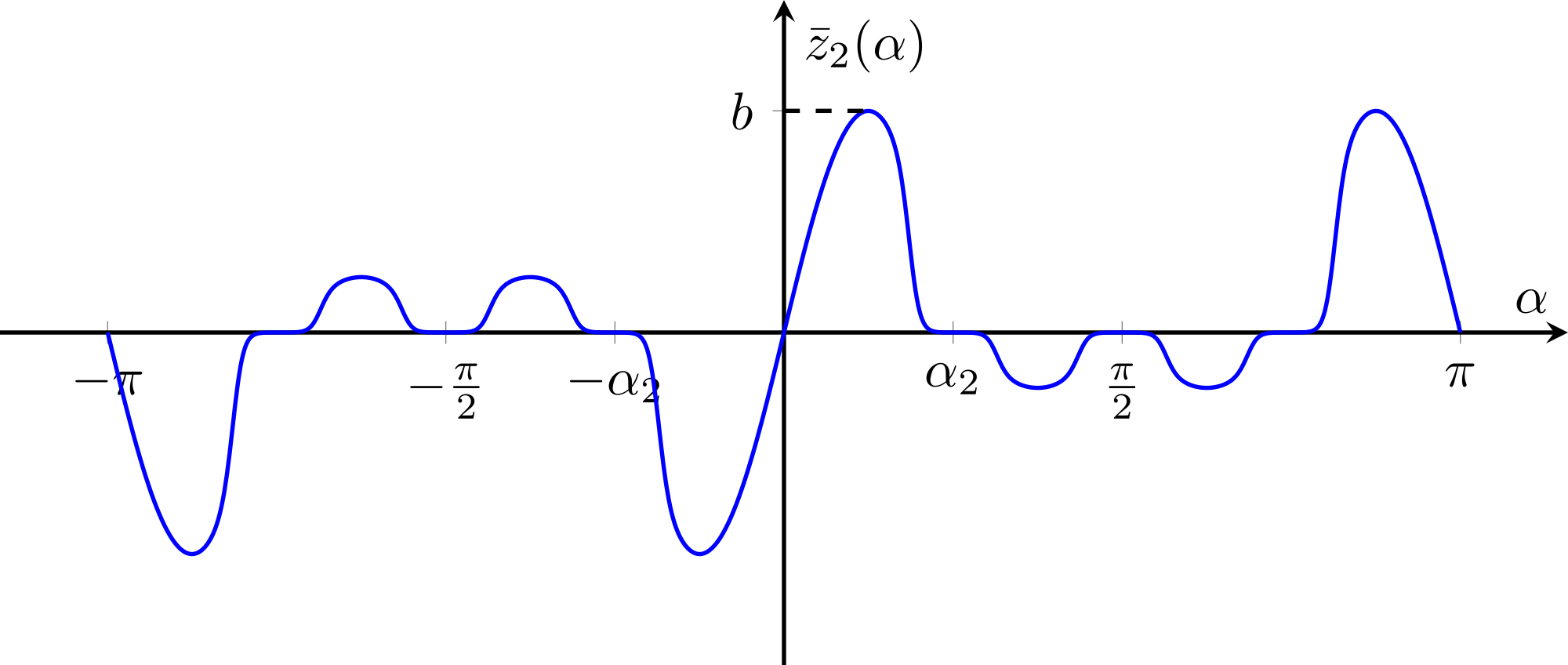}
	\caption{Sketch of a possible $\bar{z}_2(\alpha)$ fulfilling conditions 1-4.}
\end{figure}
With this definition, conditions 1, 2 and 4 are satisfied. It remains to show 3. Recalling \eqref{eq:turningint}, we further use the new symmetry defined by condition 4 to obtain 
\begin{align*}
  \partial_t \partial_\alpha \bar z_1(0,0)= \frac{\partial_\alpha \bar z_2(0)}{4\pi}\int_0^{\pi/2}\bar z_2(\beta) \sin(\bar z_1(\beta))\p_\alpha \bar z_1(\beta)
\cdot\Big(&\frac{1}{\cosh(\bar z_2(\beta))-\cos(\bar z_1(\beta))}\\
&\quad+ \frac{1}{\cosh(\bar z_2(\beta))+\cos(\bar z_1(\beta))}\Big) d\beta.
\end{align*}
Again, we split 
\begin{equation*}
     \partial_t \partial_\alpha \bar{z}_1(0,0) = K_1 + K_2,
\end{equation*}
where
\begin{align*}
    K_1 &= \partial_\alpha \bar z_2(0) \frac{1}{2\pi}  \int_{0}^{\alpha_2} \frac{\bar z_2(\beta) \sin(\bar z_1(\beta)) \cosh(\bar z_2(\beta))}{\cosh^2(\bar z_2(\beta))-\cos^2(\bar z_1(\beta))} \partial_\alpha \bar z_1(\beta) d\beta
\end{align*}
and 
\begin{align*}
    K_2 &= \partial_\alpha \bar z_2(0) \frac{1}{2\pi}  \int_{\alpha_2}^{\frac{\pi}{2}} \frac{\bar z_2(\beta) \sin(\bar z_1(\beta)) \cosh(\bar z_2(\beta))}{\cosh^2(\bar z_2(\beta))-\cos^2(\bar z_1(\beta))} \partial_\alpha \bar z_1(\beta) d\beta.
\end{align*}
At this point, we are on the same situation as in Proposition \ref{prop:turning1}, where 
$\lim_{b \to \infty} K_1  =0$ and $
K_2  < 0 $
independently of $b$. Therefore, $K_1 + K_2 <0$ for $b$ large enough, which gives condition 3. Note that the even symmetry of $\bar{z}_2$ with respect to $\alpha = \pm \frac{\pi}{2}$ means that the turning instability is also produced at $\alpha = \pm {\pi}$.
\end{proof}

\section{Numerical Simulations}\label{sec:numerics}
In this section we describe two numerical simulations of the Rayleigh-Taylor unstable gravity driven Stokes-Transport system evidencing the behavior of the Stokes-Transport system and the algorithm used to simulate them.

Before we proceed with the numerical scheme implemented, we recall the contour formulation for the Stokes-Transport problem in the case of a graph-type interface $z(\alpha,t) = (\alpha,h(\alpha,t))$ (see \cite{GGS25_SIMA})
\begin{align}
	h_t(\alpha,t) &= \frac{(\rho^- - \rho^+)}{8 \pi} \int_{-\pi}^\pi \log (2(\cosh (h(\alpha,t)-h(\beta,t))-\cos( \alpha-\beta))) h(\beta) [ 1+h_\alpha(\alpha) h_\alpha (\beta) ]   d \beta \nonumber\\
	&\quad+  \frac{(\rho^- - \rho^+)}{8 \pi} \int_{-\pi}^\pi \frac{h(\beta) (h(\alpha)-h(\beta) )}{\cosh(h(\alpha)-h(\beta)) - \cos (\alpha-\beta)} \left[ (h_\alpha(\alpha)h_\alpha(\beta)-1) \sinh(h(\alpha)-h(\beta))\right] d\beta\nonumber\\
	&\quad+  \frac{(\rho^- - \rho^+)}{8 \pi} \int_{-\pi}^\pi \frac{h(\beta) (h(\alpha)-h(\beta) )}{\cosh(h(\alpha)-h(\beta)) - \cos (\alpha-\beta)} \left[(h_\alpha (\alpha) + h_\alpha (\beta) ) \sin (\alpha-\beta)   \right] d\beta. \label{eq:grafo}
\end{align}
For simplicity, in this section we fix
$$
\frac{(\rho^- - \rho^+)}{8 \pi}=-1.
$$
In order to simulate such an equation we discretize the spatial domain $[-\pi,\pi]$ with a mesh of $M$ points, $x_j$, $j=1,...M$, and used central differences for the first spatial derivatives. Then the spatial integral is split in $M$ smaller integrals as follows. If we denote
\begin{align*}I&=\int_{-\pi}^\pi \log (2(\cosh (h(\alpha,t)-h(\beta,t))-\cos( \alpha-\beta))) h(\beta) [ 1+h_\alpha(\alpha) h_\alpha (\beta) ]   d \beta \nonumber\\
	&\quad+  \frac{(\rho^- - \rho^+)}{8 \pi} \int_{-\pi}^\pi \frac{h(\beta) (h(\alpha)-h(\beta) )}{\cosh(h(\alpha)-h(\beta)) - \cos (\alpha-\beta)} \left[ (h_\alpha(\alpha)h_\alpha(\beta)-1) \sinh(h(\alpha)-h(\beta))\right] d\beta\nonumber\\
	&\quad+  \frac{(\rho^- - \rho^+)}{8 \pi} \int_{-\pi}^\pi \frac{h(\beta) (h(\alpha)-h(\beta) )}{\cosh(h(\alpha)-h(\beta)) - \cos (\alpha-\beta)} \left[(h_\alpha (\alpha) + h_\alpha (\beta) ) \sin (\alpha-\beta)   \right] d\beta,
\end{align*}
we have that
\begin{align*}
I&\approx\sum_{j=1}^M
\int_{x_{j-1}}^{x_j} \log (2(\cosh (h(\alpha,t)-h(\alpha-\beta,t))-\cos(\beta))) h(\alpha-\beta) [ 1+\delta h(\alpha) \delta h (\alpha-\beta) ]    \nonumber\\
	&\quad+  \frac{h(\alpha-\beta) (h(\alpha)-h(\alpha-\beta) )}{\cosh(h(\alpha)-h(\alpha-\beta)) - \cos (\beta)} \left[ (\delta h(\alpha)\delta h(\alpha-\beta)-1) \sinh(h(\alpha)-h(\alpha-\beta))\right] \nonumber\\
	&\quad+  \frac{h(\alpha-\beta) (h(\alpha)-h(\alpha-\beta) )}{\cosh(h(\alpha)-h(\alpha-\beta)) - \cos (\beta)} \left[(\delta h (\alpha) + \delta h (\alpha-\beta) ) \sin (\beta)   \right] d\beta
\end{align*}
where $\delta$ denotes the central finite differences. Then, for a fixed $x_i$, two different types of integrals arise. On the one hand we have that the integrals 
$$
I_1=\int_{0}^{x_1}\text{ and }I_M=\int_{x_{M-1}}^{2\pi}
$$
are singular integrals. The singularity in these integrals are removed using a Taylor expansion. Indeed, we approximate
\begin{multline*}
I_1\approx h(\alpha)(1+(\delta h(\alpha))^2)\int_0^{x_1}\log(4\sin^2(\beta))d\beta\\+h(\alpha)(1+(\delta h(\alpha))^2)\log(1+(\delta h(\alpha))^2)\int_0^{x_1}d\beta+h(\alpha)2(\delta h(\alpha))^2)\int_0^{x_1}d\beta.
\end{multline*}
We use periodicity for $I_M$. As the remaining integrals are not singular, no further simplification is required in order to approximate them. These integrals are then computed using Simpson's rule. Then, an artificial viscosity is added
$$
\varepsilon \partial_\alpha^2 h(\alpha) \approx \varepsilon \delta^2h(\alpha)
$$
and an explicit Runge-Kutta with adaptive time step time integrator is carried out. We used the Matlab function \emph{ode45} for this task.

We used this scheme with $M=1024$ and $\varepsilon=0.001$ to compute the following examples. For the first example we consider a polygonal initial data given by
\begin{equation*}
f_1(x)=\left\{\begin{array}{cc}
x, &\text{ for } 0<x\leq1,\\
1, &\text{ for }1<x\leq\pi-1,\\
-x-\pi, &\text{ for }\pi-1<x\leq\pi,\\
\text{odd extension },&\text{ for }\pi<x\leq 2\pi.
\end{array}
\right.
\end{equation*}

The numerical results are shown in Figure \ref{fig:ex1a}, which suggest that for a short time, the length of the finger grows, whereas the curvature does not grow. The numerical simulation is only run for a short time up to $t=0.3$, since the numerical scheme we described above computes the evolution of a \emph{graph} interface $h(\alpha,t)$, and the scheme breaks down once an overturn happens.
\begin{figure}[htbp]
    
    \hspace*{-1.5cm}\includegraphics[width=11cm]{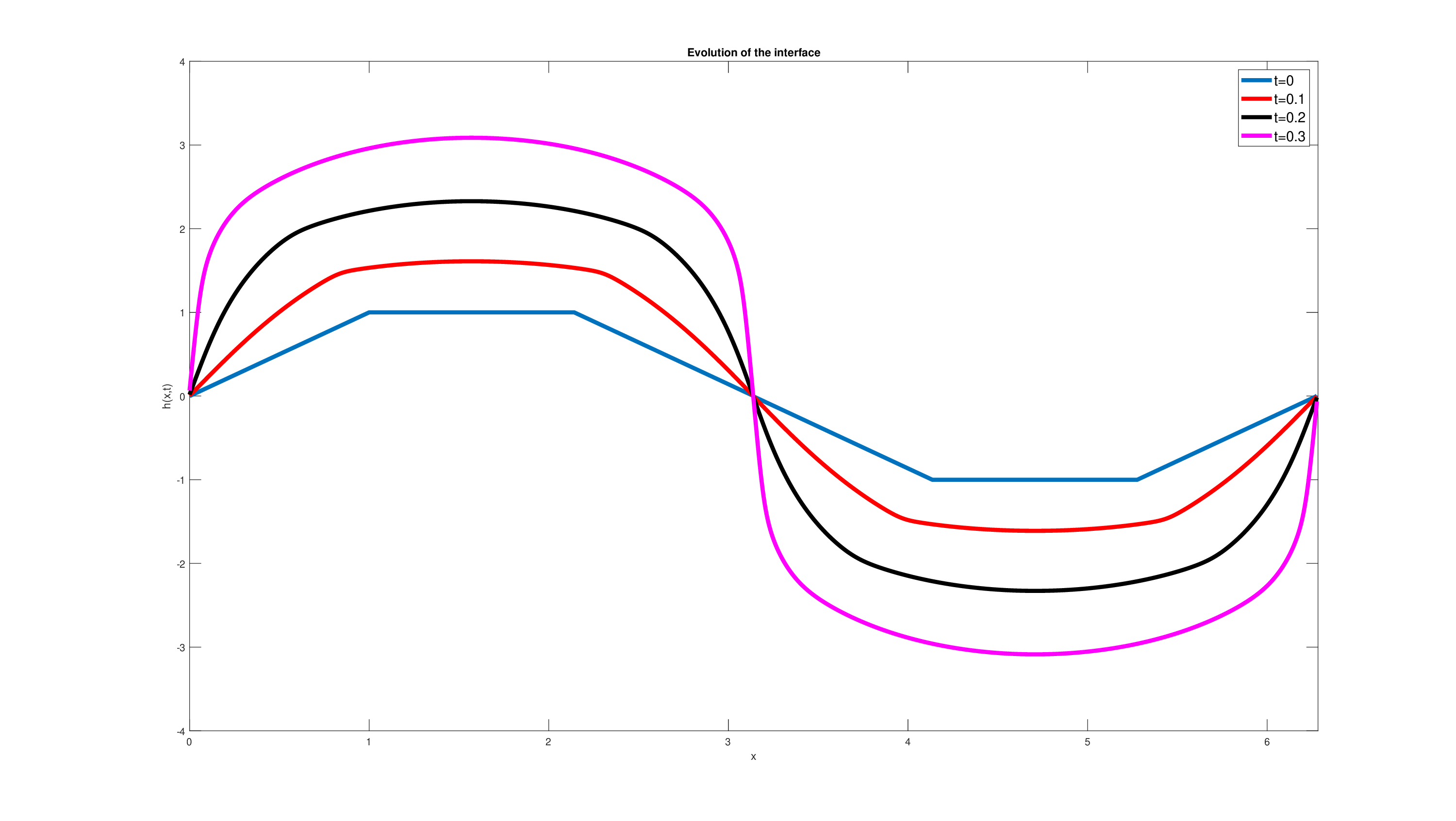} 
    \hspace*{-0.5cm}
    \includegraphics[width=8cm]{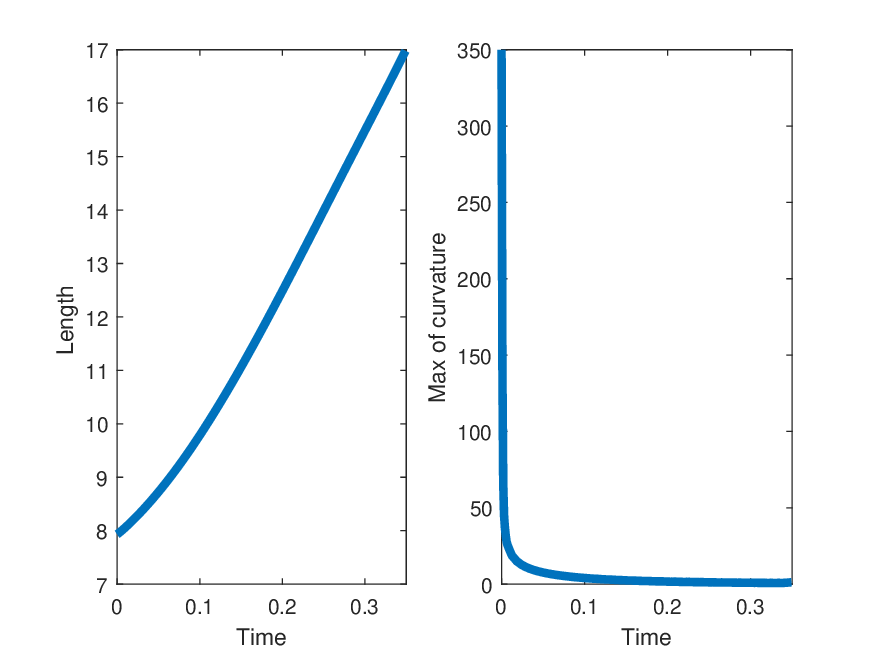}

\hspace*{4cm} (a) \hspace*{8.5cm} (b)
    
    \caption{(a) Evolution of the interface for Example 1; (b) Evolution of the length and the curvature of the interface for Example 1.}
    \label{fig:ex1a}
\end{figure}

For the second example we use the following initial data 
\begin{equation*}
f_2(x)=\left\{\begin{array}{cc}
\sin(x)^3, &\text{ for } 0<x\leq\pi,\\
-\sin(x)^3,&\text{ for }\pi<x\leq 2\pi.
\end{array}
\right.
\end{equation*}
In this case we run the code with $M=2048$. This initial data leads somehow to a \emph{relaxation} in the number of \emph{fingers} (see Figure \ref{fig:ex3}).
\begin{figure}[h]
    \centering
    \includegraphics[width=13cm]{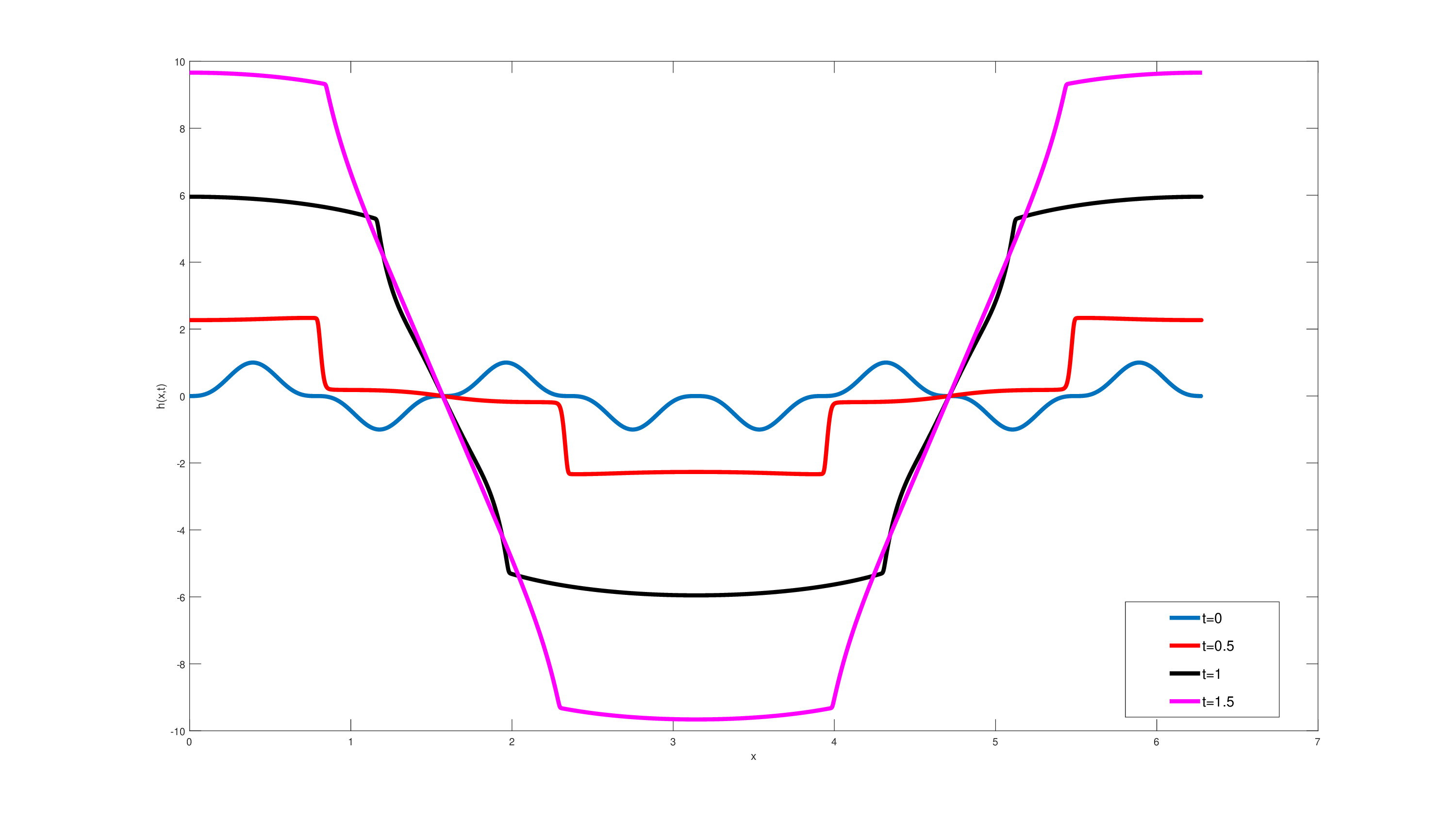}
    \caption{Evolution of of the interface for Example 2.}
    \label{fig:ex3}
\end{figure}

\newpage
\subsection*{Acknowledgments}
FG, RGB, ZH and ES are grateful to the Institute of Advanced Study at Princeton, where part of this work was done, for their warm hospitality and stimulating discussions. FG and ES were partially supported by the AEI through the grants PID2022-140494NA-I00 and PID2024-158418NB-I00.
FG was partially supported by MINECO grant RED2022-134784-T (Spain). ES acknowledges support from the Max Planck Institute for Mathematics in the Sciences and was partially supported by the Generalitat Valenciana through the call “Subvenciones a grupos de investigación emergentes”, project CIGE/2024/115. RGB was supported by the project ``An\'alisis Matem\'atico Aplicado
y Ecuaciones Diferenciales" Grant PID2022-141187NB-I00 funded by MCIN/AEI/10.13039/\\501100011033/FEDER, UE. ZH acknowledges partial support from the NSF-DMS Grant 2306726, NSF-DMS Grant 2106528, and a Simons Collaboration
Grant 601960. YY acknowledges partial support from the MOE Tier 1 grant and the Asian Young
Scientist Fellowship. The authors are grateful to Franck Sueur for stimulating discussions.


\bibliographystyle{acm}
\bibliography{references.bib}

\begin{flushleft}
	Francisco Gancedo: 
	\textsc{Departamento de An\'alisis Matem\'atico \& IMUS\\ Universidad de Sevilla,
	41012 Sevilla, Spain.}
	\textit{E-mail:
    fgancedo@us.es}\\
    \medskip
    Rafael Granero-Belinchón: 
	\textsc{Departamento de Matemáticas, Estadística y Computación\\ Universidad de Cantabria,
	39005 Santander, Spain.}
	\textit{E-mail:
    rafael.granero@unican.es}\\
    \medskip
    Zhongtian Hu: 
	\textsc{Mathematics Department,
    Princeton University\\
	Princeton, NJ 08544-1000, United States.}
	\textit{E-mail:
    zh1077@princeton.edu} \\
     \medskip
   Elena Salguero: 
	\textsc{Max Planck Institute for Mathematics in the Sciences\\
	04103 Leipzig, Germany.}
	\textit{E-mail:
    elena.salguero@mis.mpg.de} \\
    \medskip
    Yao Yao: 
	\textsc{Department of Mathematics, National University of Singapore\\
	119076, Singapore.}
	\textit{E-mail:
    yaoyao@nus.edu.sg}
\end{flushleft}

\end{document}